\documentclass[a4paper,10pt,reqno]{amsart}

\usepackage{amssymb,amsmath,amsfonts} 
\usepackage{a4wide}
\usepackage{graphicx}
\usepackage{hyperref}
\usepackage{color}
\usepackage{amsthm}
\usepackage{enumerate}
\numberwithin{equation}{section}    
\theoremstyle{plain}
\newtheorem{Theorem}{Theorem}[section]
\newtheorem{Proposition}[Theorem]{Proposition}
\newtheorem{Corollary}[Theorem]{Corollary}
\newtheorem{Lemma}[Theorem]{Lemma}
\theoremstyle{definition}

\newtheorem{Definition}[Theorem]{Definition}

\newtheorem{Examples}[Theorem]{Examples}

\theoremstyle{remark}
\newtheorem{Remark}[Theorem]{Remark}
\newcommand{\one}{\mathbf{1}} 

\newcommand{\FF}{\mathbb{F}}
\newcommand{\RR}{\mathbb{R}}

\newcommand{\NN}{\mathbb{N}}

\newcommand{\ZZ}{\mathbb{Z}}

\newcommand{\cC}{\mathcal{C}}

\newcommand{\cF}{\mathcal{F}}
\newcommand{\cP}{\mathcal{P}}
\newcommand{\cD}{\mathcal{D}}
\newcommand{\cB}{\mathcal{B}}

\newcommand{\cS}{\mathcal{S}}
\newcommand{\cA}{\mathcal{A}}
\newcommand{\cZ}{\mathcal{Z}}
\newcommand{\cG}{\mathcal{G}}

\newcommand{\cN}{\mathcal{N}}
\newcommand{\cO}{\mathcal{O}}
\newcommand{\cK}{\mathcal{K}}
\newcommand{\cQ}{\mathcal{Q}}
\newcommand{\cR}{\mathcal{R}}
\newcommand{\cJ}{\mathcal{J}}
\newcommand{\cI}{\mathcal{I}}
\newcommand{\cT}{\mathcal{T}}

\newcommand\abs[1]{\left|#1\right|}

\newcommand\inn[1]{\left\langle #1 \right\rangle}
\newcommand\set[1]{\left\{{#1}\right\}}

\newcommand{\hm}[1]{\textbf{*}\leavevmode{\marginpar{\tiny%
$\hbox to 0mm{\hspace*{-0.5mm}$\leftarrow$\hss}%
\vcenter{\vrule depth 0.1mm height 0.1mm width \the\marginparwidth}%
\hbox to 0mm{\hss$\rightarrow$\hspace*{-0.5mm}}$\\\relax\raggedright #1}}}

\begin{document}

\title[The Shannon-McMillan-Breiman theorem for group actions]{The Shannon-McMillan-Breiman theorem \\ beyond amenable groups}


\date{\today}

\author{Amos Nevo}
\address{Technion, Israel Institute of Technology}
\email{anevo@tx.technion.ac.il}

\author{Felix Pogorzelski}
\address{University of Leipzig}
\email{felix.pogorzelski@math.uni-leipzig.de}

\thanks{The first author acknowledges the support of  ISF Moked grant \# 2095/15. The second author gratefully acknowledges support through a Technion post-doctoral Fine fellowship}

\begin{abstract}
We introduce a new isomorphism-invariant notion of entropy for measure preserving actions of arbitrary countable groups on probability spaces, which we call orbital Rokhlin entropy.  It employs Danilenko's orbital approach to entropy of a partition,  and is motivated by  Seward's recent generalization of Rokhlin entropy from amenable to general groups. 
A key ingredient in our approach is the use of an auxiliary probability-measure-preserving hyperfinite equivalence relation. 
Under the assumption of ergodicity of the auxiliary equivalence relation, 
 our main result is a Shannon-McMillan-Breiman pointwise almost sure convergence theorem for the orbital entropy of partitions in measure-preserving group actions, the first such convergence result going beyond the realm of amenable groups. As a special case, we obtain a Shannon-McMillan-Breiman theorem for 
all strongly mixing actions of {\em any} countable group. Furthermore, we compare orbital Rokhlin entropy to Rokhlin entropy, and using an important recent result of Seward
we show that they coincide for free, ergodic actions of any countable group. 

Finally, we consider actions of non-Abelian free groups and demonstrate the geometric significance of the entropy equipartition property implied by the Shannon-McMillan-Breiman theorem.  We show that the orbital entropy of a partition is the limit of the information functions of the sequence of partitions arising from refining any given finite partition along almost every horoball in the group.   


\end{abstract}

\maketitle


\tableofcontents


\section{Introduction and statement of main results}\label{sec:intro}
\subsection{Shannon-McMillan-Breiman theorem for free groups} 
 The main result of this paper is the  Shannon-McMillan-Breiman entropy equipartition theorem 
 stated in Theorem~\ref{thm:MAIN_SMBhyp} below.  
To explain our main result, let us begin by stating it in a special case, namely that of probability-measure-preserving (p.m.p.) actions of the free group $\FF_r$, $ 2\le r < \infty$. We  refer to section \ref{sec:freegroup} for full discussion and proof. We recall that the Cayley graph of $\FF_r$ with respect to free generators can be identified with a $2r$-valent tree, where the group elements are identified with its vertices. We recall the associated word-metric balls $B_n(e)$ centered at the identity $e$, the boundary $\partial \FF_r$ identified with the Cantor set of infinite reduced words, and the horoballs $B^\eta$ determined by a boundary point $\eta\in \partial \FF_{2r}$, passing at the identity $e$. The set $B_{2n}\cap B^\eta:=B_{2n}^\eta$ are called the horospherical balls determined by $\eta$, and they form a sequence of increasing finite sets, exhausting the horoball $B^\eta$. We can now state the following Shannon-McMillan-Breiman pointwise convergence theorem for actions of $\FF_r$ (See Subsection~\ref{sec:SMBtheorem} for all unexplained notions.)
 \begin{Theorem}\label{free gps} 
Let $(X,\mu)$ be an essentially free p.m.p. action of the free group $\FF_r$, $2\le r< \infty$. Assume that the action is ergodic under the index $2$ subgroup of even words $\FF_r^e$. Fix any finite partition $\cP$ of $X$. For a boundary point $\eta \in \partial \FF_r$, consider the sequence of partitions obtained from $\cP$ by refining it along the horospherical balls $B_{2k}^\eta$ determined by $\eta$, namely the partitions $\bigvee_{g\in B_{2k}^\eta}g \cP$. 
Then the normalized information functions of the refined partitions converge to a constant value $h_\cP^\ast$, for almost every $x\in X$ and for almost every $\eta \in \FF_r$. 

The infimum of $h_\cP^\ast$ over all countable generating partitions of $X$ is an isomorphism-invariant of the action denoted $h^{\operatorname{OR}}\big(\FF_r \curvearrowright X\big)$, and will be called the orbital Rokhlin entropy.
 It coincides with the Rokhlin entropy and with the finitary entropy of the free group action.
\end{Theorem}

 \subsection{Historical background } 
 
 The concept of entropy for a measure-preserving transformation on a probability space was defined by Kolmogorov and Sinai, who introduced it as an isomorphism invariant for dynamical systems \cite{Ko58,Ko59,Si59}. For a fixed finite partition of the probability space Breiman \cite{Br57} proved pointwise almost-everywhere convergence of the information function associated with the partition, thus establishing asymptotic "entropy equipartition" under the transformation. Incorporating previous results of Shannon \cite{Sh48} and McMillan \cite{Mc53}, pointwise almost-everywhere entropy equipartition is usually referred to as the {\em Shannon-McMillan-Breiman theorem}. 
 
 The generalization of Kolmogorov-Sinai entropy to p.m.p.\@ actions of countable amenable groups was developed by several authors, including Kieffer \cite{Ki75}, Ollagnier \cite{Ol85} and Ornstein-Weiss \cite{OW87}. In particular, Ornstein-Weiss  raised the natural question of extending the entropy equipartition theorem to actions of amenable groups, and proved an analog of the Shannon-McMillan-Breiman theorem along regular 
 F{\o}lner sequences in amenable groups. For ergodic actions, this result was generalized by Lindenstrauss to  tempered F{\o}lner sequences \cite{Li01} (satisfying a very mild growth condition). The method of proof developed in \cite{Li01} motivate the techniques developed and used in the present paper. In the survey paper \cite{W03} Weiss gives a combinatorial proof in the case of possibly non-ergodic actions and for general tempered F{\o}lner sequences. 
 
 The development of entropy theory for p.m.p.\@ actions of non-amenable groups has been a long-standing open problem, which has seen revolutionary developments in the past decade, initiated by the breakthrough work on sofic entropy by Bowen \cite{Bo10c, Bo12}, and its refinements by  Kerr and Li \cite{KL11, KL13}. In a separate very important recent development, for ergodic p.m.p.\@ actions of general countable groups Seward introduced the notion of {\em Rokhlin entropy} \cite{Se15a,Se15b}. In \cite{Se16} Seward establishes a remarkable inequality which implies that Rokhlin entropy coincides with another notion of entropy, which we will call finitary entropy.  Seward's inequality will play a crucial role in our considerations below. 
 
 These developments naturally raise the question of the validity of the Shannon-McMillan-Breiman theorem in the far more general context of actions of non-amenable groups. Motivated by an approach developled by Danilenko and Park for actions of amenable groups \cite{Da01, DP02}, we define the {\em orbital entropy} for a partition, for any ergodic p.m.p.\@ action of any countable group. With these values at hand, we define the of concept {\em orbital Rokhlin entropy} of the underlying group action. We will show using Seward's inequality that it fact coincides with Rokhlin entropy for free actions, cf.\@ Proposition~\ref{thm:ENT}. Most importantly, for orbital entropy, we establish a general Shannon-McMillan-Breiman theorem, cf.\@ Theorem~\ref{thm:MAIN_SMBhyp}.

Let us now turn to describe the general context and the main results. 

\subsection{Probability measure preserving actions of groups and cocycle extensions}

Throughout the paper, we consider standard Borel equivalence relations with countable classes 
$\cR\subset Y\times Y$, where $(Y,\nu)$ is a probability space. $[y]=\cR(y)\subset Y$ denotes the $\cR$-class of $y\in Y$. We assume that the equivalence relation preserves the probability measure $\nu$. It is further assumed that $\cR$ is hyperfinite, or equivalently,  that $\cR$ is amenable in the sense of \cite{CFW81}.  
Thus  $\cR$ can be 
written as an increasing union of equivalence subrelations $\cR_n\subset \cR$, each 
with finite classes, i.e.
\[
\cR(y) := \bigcup_{n=1}^{\infty} \cR_n(y),\,\text{ for $\nu$-almost every } y\in Y.
\] 
Such a representation of $\cR$ 
 will be called a {\em hyperfinite exhaustion}.  If for every $n$, the relations $\cR_n$ are bounded, namely 
 the size of the equivalence classes $\cR_n(y)$ is essentially bounded (with a bound depending on $n$), we will call the representation a 
 {\it bounded hyperfinite exhaustion}. Let  us note that hyperfinite relations always admit a bounded hyperfinite exhaustion, as will be seen below. In fact it is possible to construct a hyperfinite exhaustion satisfying $\abs{\cR_n(y)}\le n$ almost surely, as noted in \cite{W84} in a more general context. 

Throughout the paper, $\Gamma$ denotes a countable group. 
The collection of all finite subsets of $\Gamma$ is denoted
by $Fin\big( \Gamma \big)$. We consider a probability measure preserving (p.m.p.) ergodic action 
$\Gamma \curvearrowright (X,\lambda)$, and aim to define a notion of entropy for the action.

To that end, assume that there is an amenable p.m.p.\@ equivalence relation $\cR$
over $(Y,\nu)$, admitting a measurable cocycle $\alpha : \cR\to \Gamma$. 
Thus for $\nu$-almost every $y \in Y$ and every $w,u,z \in [y]=\cR(y)$, 
the cocycle identity holds :
\[
\alpha\big(z,u\big) = \alpha\big(z,w \big)\cdot \alpha\big(w,u \big).
\]
A crucial construction in our discussion is the equivalence relation, denoted $\cR^X$, which is the extension of the equivalence relation $\cR$ by the cocycle $\alpha$ and the $\Gamma$-action on $X$. 
The {\em extended equivalence relation $\cR^X$}  over
$\big( X \times Y, \lambda \times \nu \big)$ is defined by the condition 
\[
\big( (x,y), (x^\prime,y^\prime) \big) \in \cR^X \Longleftrightarrow y\cR y^\prime \,\,\text{ and } x=\alpha(y,y^\prime)x^\prime
\,.\]

Assuming that the measure $\nu$ is $\cR$-invariant, it follows that $ \lambda \times \nu$ is $\cR^X$-invariant,
since the $\Gamma$-action on $X$ preserves $\lambda$. The projection map $\pi : \cR^X\to \cR$ given by $(x,y)\to y$ is {\it class injective}, namely injective on almost every $\cR^X$-equivalence class.

Further, it 
is well known that an extension of an amenable action is amenable, and thus in particular if $\cR$ is amenable, so is $\cR^X$. Thus, when $\cR$ is hyperfinite, so is $\cR^X$.  But since the extension is class-injective, in fact every hyperfinite exhaustion $(\cR_n)$ of $\cR$ can be canonically lifted to a hyperfinite exhaustion $(\cR_n^X)$ of $\cR^X$, via $\cR_n^X((x,y))=\set{(\alpha(z,y)x,z)\,;\, z\in \cR_n(y)}$. Note that if $(\cR_n)$ is a bounded hyperfinite exhuastion, then so is $(\cR_n^X)$, with the same bounds on the equivalence classes.

The cocycle $\alpha$ is called {\it class  
injective} if for $\nu$-almost every $y \in Y$, 
we have that $\alpha(y,z) \neq \alpha(y,w)$ whenever $w\neq z$.
In order to avoid degenerate cases, we will assume in the
sequel that the cocycle under consideration is class injective.

\subsection{Definition of orbital Rokhlin entropy}

In order to define the measure theoretic orbital entropy of the p.m.p.\@ $\Gamma$-action on $X$, let us begin by recalling the following construction, due to Danilenko \cite{Da01}. First, for a 
countable measurable partition 
$\cP = \{A_i\,;\,i\in I\}$ of $X$ the Shannon entropy $H(\cP)$ 
is defined as
\[
H\big( \cP \big):= - \sum_{A \in \cP} \lambda(A)\,\log\big( \lambda(A) \big),
\] 
where we use the convention that $0\cdot \log\,0 = 0$.  
For two countable partitions $\cP$ and $\cQ$, their common refinement
is $\cP \vee \cQ:= \big\{ P \cap Q\,|\, P \in \cP,\, Q\in \cQ \}$. 
For a finite set $F \in Fin(\Gamma)$, we set
\[
\cP^F := \bigvee_{g \in F} g^{-1}\,\cP,
\] 
where $g^{-1}\cP=\set{g^{-1}A_i\,;\, i\in I}$. If $F$ is the empty set, then we define $\cP^F$ to be the trivial partition, which of course has Shannon entropy zero. 
Given two partitions $\cP$ and $\cQ$,  
$\cQ$ is called {\em finer than} $\cP$ or a {\em refinement of} $\cP$, denoted 
$\cQ \geq \cP$, if for every $Q \in \cQ$, there is some $P \in \cP$ such 
that $Q \subseteq P$ up to $\lambda$-measure zero.

Now for a subset function $\cF: Y \to Fin(Y)$ (as always, with $\cF(y)\subset \cR(y)$ a.e.) and a countable 
partition $\cP$ of $X$ with $H(\cP) < \infty$, we consider the entropy function, defined by Danilenko in \cite{Da01} : 

 
\[
h^{\cP}(\cF): Y \to [0,\infty): h^{\cP}(\cF)(y) := 
H\Big( \bigvee_{z \in \cF(y)} \alpha(z,y)^{-1}\,\cP \Big).
\]

In was noted by Danilenko in \cite[Cor. 2.7]{Da01} that one can attach a notion of orbital entropy to every partition $\cP$ with 
finite Shannon entropy, and that the following holds.  

\begin{Theorem}[see {\cite[Cor. 2.7]{Da01}} and {\cite[Prop. 1.2]{DP02}}]  \label{thm:MAINcocycleentropy}
Given a class injective cocycle on $\cR$,   
for every countable partition $\cP$ of $X$ with $H(\cP) < \infty$, there is a number $h_{\cP}^{*}(\alpha)$
such that for every bounded hyperfinite exhaustion $(\cR_n)$,
\begin{eqnarray*}
h_{\cP}^{*}(\alpha) := \lim_{n \to \infty} \int_Y \frac{h^{\cP}\big( \cR_n \big)(y)}{\big| \cR_n(y)\big|}\, d\nu(y).
\end{eqnarray*} 
\end{Theorem}
Let us highlight the crucial fact that it is part of the conclusion of Theorem \ref{thm:MAINcocycleentropy} that the limit is independent of the choice of the bounded hyperfinite exhaustion.
\begin{Remark}
As can be seen from the proof of Theorem \ref{thm:MAINcocycleentropy}, the  
assumption that the cocycle being class injective is not necessary for this convergence result. 
 But in our discussion below we will dispense
with this additional generality and focus on the set-up explained above of free p.m.p. actions of countable groups. The injectivity assumption then makes sure that the values $h_{\cP}^{*}(\alpha)$ accurately reflect entropy theoretic information regarding the action of $\Gamma$ on $X$. 

Furthermore, the results we develop below can be formulated more generally, for all class-injective extensiosn of arbitrary hyperfinite relations. But again, we will focus below on our main goal, which is the development of entropy theory for p.m.p. actions of non-amenable groups.
\end{Remark}

Following \cite{Da01}, the number $h_{\cP}^{*}(\alpha)$ shall be called {\em the orbital entropy} of the
partition $\cP$ for the action $\Gamma \curvearrowright (X,\lambda)$. Danilenko \cite[Def. 2.4]{Da01} proceeds to define the orbital entropy of the associated dynamical systems as the supremum of $h_{\cP}^{*}(\alpha)$ over all finite partitions.

We propose a different definition, motivated by the developments in entropy theory for non-amenable groups due to Bowen and to Seward, particularly by Seward's recent work on Rokhlin entropy (see the discussion below). To define  {\it measure-theoretic entropy for the dynamical system} $\Gamma \curvearrowright (X,\lambda)$ itself, we restrict ourselves to {\em generating partitions},
i.e.\@ countable partitions $\cP$ for $X$ such that (modulo null sets) 
\[
\bigvee_{g \in \Gamma} g \, \cP = \cB\,.
\]

\begin{Definition}[Orbital Rokhlin entropy]
 Let $(X,\lambda)$ be a p.m.p.\@ ergodic action of $\Gamma$.
Let $\alpha: \cR \to \Gamma$ be a class injective cocycle defined on a hyperfinite relation. The number 
\begin{eqnarray*}
h^{\operatorname{OR}}(\Gamma \curvearrowright X) := \inf\big\{h^{*}_{\cP}(\alpha)\,| \, \cP \mbox{ countable, generating partition}\big\}
\end{eqnarray*}
is called the {\em orbital Rokhlin entropy} for the action $\Gamma \curvearrowright (X,\lambda)$ and for the injective cocycle $\alpha: \cR \to \Gamma$.  
\end{Definition}

As we shall see presently, the orbital Rokhlin entropy of a free ergodic action is an intrinsic invariant, independent of the p.m.p.\@ hyperfinite relation $\cR$, the bounded hyperfinite exhaustion $\cR_n$, and the class injective cocycle $\alpha$ used to define it. 
This remarkable fact is ultimately based on the important recent result established by Seward relating two completely different notions of entropy, which we now proceed to describe.

\subsection{Orbital Rokhlin entropy, finitary entropy and Rokhlin entropy}

The notion of Rokhlin entropy has its origin in Rokhlin's studies of entropy for a single transformation. It has recently been greatly generalized and studied intensively
in the context of general group actions by Seward in the series of papers \cite{Se15a, Se15b, Se16}. 
In particular, Seward defines  
\begin{Definition}[Rokhlin entropy]
Let $\Gamma \curvearrowright (X,\lambda)$ be a p.m.p.\@ ergodic group action. Then, the number
\[
h^{\operatorname{Rok}}(\Gamma \curvearrowright X) := \inf\big\{ H(\cP)\,|\, 
\cP \mbox{ countable, generating partition}\big\}
\]
is called the {\em Rokhlin entropy} of the group action. 
\end{Definition}

Seward and Tucker-Drob showed in \cite{ST12}
that for free ergodic actions of amenable groups, Rokhlin entropy coincides
with the classical Kolmogorov-Sinai entropy. 
In a recent important breakthrough  Seward \cite{Se16} established the following upper bound for the Rokhlin entropy for every countable, generating partition $\cP$, assuming the $\Gamma$ action on $X$ is essentially free. 
\begin{eqnarray}\label{thm:seward}
h^{\operatorname{Rok}}\big( \Gamma \curvearrowright X \big) \leq 
\inf_{T \in Fin(\Gamma)} \frac{1}{|T|}\, H\Big( \bigvee_{g \in T} g^{-1}\cP \Big).
\end{eqnarray}
(In fact, Theorem~1.5 in \cite{Se16} shows a stronger statement
involving entropy conditioned on $\Gamma$-invariant $\sigma$-algebras.)

Let us now note that by the standard subadditivity property of Shannon entropy,  
the right hand side of (\ref{thm:seward}) is bounded from above by $H(\cP)$. 
Hence, by passing to the infimum over all generating partitions, we obtain 
in fact equality of the lower and the upper bound. Let us therefore define the following notion of {\it finitary entropy}.

\begin{Definition}[Finitary entropy]
\[
h^{\operatorname{fin}}\big( \Gamma \curvearrowright X \big)
:= \inf\Big\{ \inf_{T \in Fin(\Gamma)} H\Big( \bigvee_{g \in T} g^{-1}\cP \Big) /|T| \, \,\Big|\, 
\cP \mbox{ countable, generating partition}\Big\}.
\]
\end{Definition}
By the preceding discussion, it follows from Seward's inequality (\ref{thm:seward}) that these two notions of entropy coincide for ergodic (essentially) free actions: $h^{\operatorname{fin}}\big( \Gamma \curvearrowright X \big)=h^{\operatorname{Rok}}(\Gamma \curvearrowright X)$.

Using this result, it is easy to see that our definition of orbital Rokhlin 
entropy above gives rise to the same value as well.

\begin{Proposition} \label{thm:ENT}
Assume that $\Gamma \curvearrowright (X,\lambda)$ is an ergodic essentially free p.m.p.\@ action. Then,
for every amenable p.m.p.\@ equivalence relation $\cR$ over $(Y, \nu)$ and every class injective cocycle 
$\alpha: \cR \to \Gamma$, we obtain
\begin{eqnarray*}
h^{\operatorname{fin}}\big( \Gamma \curvearrowright X \big)=h^{\operatorname{OR}}\big(\Gamma \curvearrowright X, \alpha\big) = h^{\operatorname{Rok}}\big(\Gamma \curvearrowright X\big).
\end{eqnarray*}
\end{Proposition}

\begin{proof}
Let $\cP$ be a countable partition with finite Shannon entropy. Assume that $(\cR_n)$
is any bounded hyperfinite exhaustion for $\cR$. Since the 
cocycle $\alpha$ is class injective, and $\cR_n(y)$ is a finite set almost surely, we obtain 
\begin{eqnarray*}
\inf_{T \in Fin(\Gamma)} \frac{1}{|T|}\,H\Big( \bigvee_{g \in T} g^{-1}\cP \Big)
\leq \frac{h^{\cP}\big( \cR_n \big)(y)}{\big| \cR_n(y) \big|} \leq H\big( \cP \big)\,,
\end{eqnarray*}
for each $n \in \NN$ and for $\nu$-almost every $y \in Y$.
Note that these inequalities remain valid even if $H(\cP) = \infty$.
In the case that there exist generating partitions with finite Shannon 
entropy, we 
integrate over $Y$ and pass to the limit as $n\to \infty$. By Theorem \ref{thm:MAINcocycleentropy} we conclude that $ h^{\operatorname{fin}}(\Gamma \curvearrowright X)\le h^\ast_\cP(\alpha)\le H(\cP)$
independently of the cocycle $\alpha$. Now taking the infimum over all generating partitions,  since by the discussion above 
$h^{\operatorname{Rok}}(\Gamma \curvearrowright X) = h^{\operatorname{fin}}(\Gamma \curvearrowright X)$, the above inequality immediately implies equality of all three notions of entropy. 

\end{proof}

\subsection{The Shannon-McMillan-Breiman theorem} \label{sec:SMBtheorem}
The generalization that we propose of the classical Shannon-McMillan-Breiman theorem concerns pointwise almost everywhere convergence of a sequence of natural {\em information functions} on $X$, the space on which $\Gamma$ acts by measure-preserving transformations. For
a given countable partition $\cP$ of $X$ with $H(\cP) < \infty$, we set
\[
\cJ(\cP)(x):= - \log\lambda\big( \cP(x) \big),
\] 
where we define $\cP(x)$ to be the unique set $A \in \cP$ containing 
the point $x$ (namely the $\cP$-{\em name of} $x$).
Note that by definition, we have 
\[
H\big( \cP \big) = \int_X \cJ\big( \cP \big)(x)\,d\lambda(x)\,.
\]
Given an ergodic
p.m.p.\@ group action $\Gamma \curvearrowright (X,\lambda)$, and a class-injective cocycle $\alpha : \cR \to \Gamma$ consider the 
extended relation $\cR^X$ over $X \times Y$. 
For the proof of the Shannon-McMillan-Breiman theorem, we will 
assume that the extended relation $\cR^X$ is {\em ergodic}. 
Though being a non-trivial assumption on the 
equivalence relation and the action under consideration, ergodicity of the extension is satisfied in 
many situations. One important example is that of arbitrary ergodic actions of irreducible lattices in connected semisimple Lie groups with finite center. In general, a sufficient condition in order
to guarantee ergodicity of $\cR^X$ is {\em weak mixing} of the relation
$\cR$ over $(Y,\nu)$, as defined in \cite{BN13a}.
  For a more detailed elaboration
of these issues, we refer the reader to Section  7. 

We are now
able to state our main result, namely the Shannon-McMillan-Breiman pointwise convergence theorem for actions of non-amenable groups.

\begin{Theorem}
\label{thm:MAIN_SMBhyp}
Let $\Gamma \curvearrowright (X,\lambda)$ be an ergodic essentially free p.m.p.\@ action. 
Assume that $\cR$ is an amenable, p.m.p.\@ equivalence
relation, and $\alpha : \cR\to \Gamma$ is a class injective cocycle, such that the extended relation $\cR^X$ is ergodic. 
Then, for every bounded hyperfinite exhaustion $(\cR_n)$ satisfying the growth condition
\[
\lim_{n \to \infty} \operatorname{ess}\,\operatorname{inf}_y |\cR_n(y)|/\log\,n = \infty 
\]
the information functions satisfy the following convergence property. 
Given a finite partition $\cP$ of $X$, for $(\lambda \times \nu)$-almost every $(x,y) \in X \times Y$, 
\[
\lim_{n \to \infty} \frac{\cJ\big( \cP^{\cR_n(y)}(x) \big)}{|\cR_n(y)|} = h^{*}_{\cP}(\alpha),
\]
where $h^{*}_{\cP}(\alpha)$ is the orbital entropy of the partition
$\cP$. 
\end{Theorem}
Given {\it any } countable group $\Gamma$, a suitable equivalence relation $\cR$ and cocycle $\alpha : \cR \to \Gamma$ can be found, such that Theorem 
\ref{thm:MAIN_SMBhyp} applies to {\it all} strongly mixing action of $\Gamma$. We will discuss this matter in detail in  \S \ref{kechris-Vaes} below. 

Let us note the following comments regarding the assumptions of Theorem \ref{thm:MAIN_SMBhyp}.

1)  the growth condition on the $(\cF_n)$ is very mild, and in fact one can always find bounded hyperfinite exhaustions
which satisfy it. To see this, recall that any two ergodic p.m.p.\@ actions of amenable
groups are orbit equivalent. This fact goes back to Dye \cite{Dy59} for a pair of (ergodic) 
$\ZZ$-actions and was stated in full generality in \cite{OW80}. For a survey on the
topic of orbit-equivalence, see also \cite{Ga00}. Now in \cite{CFW81} it was shown that any amenable equivalence relation is hyperfinite and any hyperfinite relation is generated by the action of a single transformation. This action is orbit equivalent to the standard odometer action, and we can transfer the  hyperfinite exhaustion of the odometer to the underlying equivalence relation (using the orbit equivalence). It follows that bounded hyperfinite
exhaustions satisfying the growth condition required in Theorem~\ref{thm:MAIN_SMBhyp} always exist. 

2) The proof of Theorem~\ref{thm:MAIN_SMBhyp} is inspired by Lindenstrauss' proof of the Shannon-McMillan-Breiman theorem for amenable groups.
As such, the analogous growth condition for (tempered) F{\o}lner sequences appears there as well, cf.\@ 
Theorem~1.3 in~\cite{Li01}.
In the survey paper \cite{W03} Weiss gives a deterministic combinatorial proof of the Shannon-McMillan-Breiman theorem based on previous joint work of Ornstein and Weiss. 
This proof is valid for general tempered
F{\o}lner sequences even without an additional growth condition and also for non-ergodic
actions, see Theorem~6.2 in \cite{W03}.

3) The existence of a hyperfinite equivalence relation $\cR$ and a cocycle $\alpha$ which satisfy the properties required in Theorem \ref{thm:MAIN_SMBhyp} {\it for all} p.m.p. actions of $\Gamma$  is an interesting question, which at this point has not been fully resolved for all ergodic p.m.p. actions of all countable groups $\Gamma$, as far as we are aware. We shall establish in \S \ref{kechris-Vaes} below, using an argument developed by A. Kechris for this purpose,  that given {\it any} countable group, it is possible to find an equivalence relation $\cR$ and a cocycle $\alpha : \cR \to \Gamma$, such that the assumption of Theorem \ref{thm:MAIN_SMBhyp} are satisfied for {\it all} strongly mixing actions of $\Gamma$. Furthermore, using a recent result of Vaes and Wahl \cite{VW17}, we shall see that for an extensive class of group  it is possible to find for {\it any} p.m.p. measure-preserving action $(X,\mu)$ of $\Gamma$,  a suitable equivalence relation $\cR^X$ over a product space $(X \times Y, \mu \times \nu)$, and a class injective cocycle $\alpha : \cR^X \to \Gamma$ satisfying all the assumptions of Theorem \ref{thm:MAIN_SMBhyp}. However, in this construction, $\cR^X$ is not derived as an extended relation coming from an equivalence relation over $(Y, \nu)$. 

 Theorem~\ref{thm:MAIN_SMBhyp} implies the corresponding 
Shannon-McMillan theorem, asserting convergence of the information functions in $L^1$. This result was originally proved by Danilenko and Park in \cite{DP02}.  

\begin{Corollary}\cite[Prop. 1.2]{DP02} \label{cor:MAIN_SML1}
 Convergence in Theorem \ref{thm:MAIN_SMBhyp} holds
for $\nu$-almost every $y \in Y$ in the $L^1(X,\lambda)$-norm, and also  
in the $L^1(X \times Y, \lambda \times \nu)$-norm. 
\end{Corollary}
We do not know whether
one can also expect convergence in 
$L^1(Y,\nu)$ for $\lambda$-almost every $x \in X$. 

\subsection*{ Acknowledgements.}
The authors would like to thank 
 Lewis Bowen, Brandon Seward and Benjamin Weiss for 
several enlightening and useful conversations and for their comments on a preliminary version of the present manuscript.  
The authors are indebted to Alexander Kechris for several insightful observations regarding measurable equivalence relations, and for his permission to include these observations and their proofs in the present paper.


\section{Amenable equivalence relations} \label{sec:amenable}

In this section, we discuss measured Borel 
equivalence relations which are amenable in the sense of 
Connes, Feldman and Weiss \cite{CFW81}. This condition was shown in that paper to be equivalent to hyperfiniteness.

\subsection{Measurable equivalence relations}
Consider a Borel measurable equivalence relation $\cR$ defined
over a standard Borel probability space $(Y, \mathcal{B}(Y), \nu)$, namely  $\cR$ is a Borel measurable subset
of $Y \times Y$ with the properties
\begin{itemize}
\item $(y,y) \in \cR$ for all $y \in Y$,
\item if $(y,z) \in \cR$, then $(z,y) \in  \cR$ for all $y,z \in Y$,
\item if $(y,z) \in \cR$ and $(z,w) \in \cR$, then $(y,w) \in \cR$ for all $y,z,w \in Y$.
\end{itemize}
Two points $y$ and $z$ are called $\cR$-equivalent points if $(y,z) \in \cR$,  and the equivalence class is denoted $\cR(y)=[y]=[y]_\cR$. Each $y\in Y$ determines a left and a right fiber in $\cR$, given by 
 $\cR^y=\set{(y,z)\,;\, z\cR y}\subset \cR$ and $\cR_y=\set{(z,y)\,;\, z\cR y}\subset \cR$. 
We will always assume that for almost every $y$, the fiber $\cR^y$ (and hence also $\cR_y$) 
is countable. $c^y$ will denote the counting measure on $\cR^y$, and $ c_y$ the counting measure on $\cR_y$. Integrating the counting measures
over $Y$, we obtain two $\sigma$-finite (but in general not finite) measures on $\cR$, namely 
$\tilde{\nu}_l=\int_Y c^y d\nu(y)$ and $\tilde{\nu}_r=\int_Y c_yd\nu(y)$. The measure $\nu$ on $Y$ is called $\cR$-non-singular if these two measures are equivalent. Note that $(\cR,\tilde{\nu}_l)$ is a standard Borel space, and $\pi_l:\cR\to Y$ given by $\pi_l(y,z)=y$ is a measurable factor map, and similarly for $\tilde{\nu}_r$ and $\pi_r$.   Note that under the coordinate projection $\pi_l :\cR\to Y$,  
the integral above expresses the measure disintegration of $\tilde{\nu}_l$ with respect to $\nu$. 
If $\tilde{\nu}_l=\tilde{\nu}_r$ then we denote it by $\tilde{\nu}$, and then $\tilde{\nu}$ as well as $\nu$  are called {\it $\cR$-invariant}, and  $\cR$ is  called a probability-measure-preserving (p.m.p.) equivalence relation. This is the only case we will consider below.

An {\em inner automorphism} of the relation $\cR$ is a measurable mapping
$\phi: Y \to Y$ which is almost surely bijective with measurable inverse
and with its graph $\operatorname{gr}(\phi)$ being contained in $\cR$.
The collection of all inner automorphisms gives rise to a group
$\operatorname{Aut}(\cR)=[\cR]$, called the {\em full group} of $\cR$. 
A countable subset $\Phi_0 \in \operatorname{Aut}(\cR)$ is said to be {\em generating}
(for $\cR$) if for $\tilde{\nu}$-almost all $(y,z) \in \cR$ there is some $\phi \in \Phi_0$
such that $z= \phi(y)$. 
$\Phi_0$ of course generates a {\em countable} subgroup of $\operatorname{Aut}(\cR)$, denoted $\Phi$.

The space $Fin(Y)$ of finite subsets of a Borel space $Y$ is a Borel space in a natural way, using the obvious Borel structure on $\bigcup_{n\in \NN}Y^n/Sym(n)$. 
Measurable mappings of the form $\cF: Y \to Fin(Y)$ satisfying that for almost every $y$, the set $\cF(y)$ 
consists of finitely many points equivalent to $y$, are called 
{\em subset functions} of $\cR$. The possibility that $\cF(y)=\phi$ is the empty set is allowed. We consider two subset functions $\cT, \cS$ to be equal if 
the set $\{y \,|\, \cS(y) \neq \cT(y)\}$ has zero measure. We write $\cS\subset \cT$ if $\cS(y)\subset \cT(y)$ for $\nu$-almost every $y$. 
 Subset functions
can be composed with each other, inverted, and subtracted from each other. We refer to \cite{BN13a} and \cite{BN13b} for a full discussion, and recall here the following definitions. 
\begin{eqnarray*}
\cS \circ \cT(y) &:=& \bigcup_{z \in \cT(y)} \cS(z), \\
\cT^{-1}(y)  &:=& \big\{ z \in [y] \,|\, y \in \cT(z)  \big\}, \\
\big( \cS \setminus \cT \big) (y) &:=& \cS(y) \setminus \cT(y). 
\end{eqnarray*} 
A finite non-empty set $D \subset \operatorname{Aut}(\cR)$ gives rise to the subset function 
$\cD(y)=\set{\phi(y)\,;\, \phi\in D}$. Given a 
subset function $\cT$ and a finite set $D \subset \operatorname{Aut}(\cR)$, $\cD\circ \cT$ is defined as above, and is given by 
\[
\cD \circ\cT(y) := \bigcup_{\phi \in D} \phi\big( \cT(y) \big) .
\]
We will also use the notation $D\circ \cT$ for this expression. 
A subset $\cK \subseteq \cR$ is said to be {\em bounded} if  
\[
\|\cK\|:= \operatorname{ess-sup}_y \max \big\{ |\cK_y|, |\cK^y| \big\}< \infty,
\]
where $\cK_y := \big\{ z \in [y] \,|\, (z,y) \in \cK \big\}$
and $\cK^y:= \big\{ z \in [y]\,|\, (y,z) \in \cK \big\}$.
Analogously, we say that a subset function $\cT$ is {\em bounded}, if
$\|\cT\| := \operatorname{ess-sup}_y \max\, \big\{ |\cT(y)|, |\cT^{-1}(y)| \big\}$
is finite.

\subsection{Hyperfiniteness and amenability}

The relation $\cR$ is called hyperfinite if there exists a
sequence $(\cR_n)$ consisting of subrelations $\cR_n\subset \cR$, where each $\cR_n$ has finite
classes, satisfying 
\[
\cR_n \subseteq \cR_{n+1}\,\,\,,\,\,\,\cR = \bigcup_{n=1}^{\infty} \cR_n.
\] 
We refer to such a sequence as a {\em hyperfinite exhaustion} of
$\cR$. Note that each $\cR_n$ is a subset function as defined above. 
If each $\cR_n$ is a bounded subset function we will call the hyperfinite exhaustion a {\it bounded hyperfinite exhaustion}. 
Recall that it was proved by Connes, Feldman and Weiss  \cite[Thm. 10]{CFW81} that $\cR$ being amenable
is equivalent to $\cR$ being {\em hyperfinite}.

Let us note that 
hyperfinite exhaustions are asymptotically invariant under inner automorphisms of finite rank, in the following sense (see also \cite{BN13a}, \cite{BN13b}).

\begin{Proposition} \label{prop:finitephin}
If $(\cR_n)$ is a bounded hyperfinite exhaustion of the relation $\cR$, then there 
exists an increasing sequence of finite subgroups $\Phi_n \subseteq \operatorname{Aut}(\cR)$, $n \geq 1$,
such that $\Phi := \bigcup_{n \geq 1} \Phi_n$ is generating for $\cR$ 
and for all $n \in \NN$ and $\phi \in \Phi_n$, the graph $\operatorname{gr}(\phi)$
is contained in $\cR_n$.

In particular, for every $\phi \in  \Phi $,
there is an $n_0 \in \NN$ such that $\cR_n(y) \, \triangle \, \phi(\cR_n(y)) = \emptyset$ 
for $\nu$-almost every $y \in Y$ and every $n \geq n_0$. 
\end{Proposition}  
\begin{proof}
Let $\cT$ be any equivalence subrelation of $\cR$ with finite classes of bounded size. We can divide $Y$ to finitely many measurable $\cT$-invariant sets where the size of $\cT(y)$ is fixed, and we can restrict $\cT$ to one of them. Without loss of generality we can thus assume that $\cT$ has classes of fixed size $N$ in $Y$.  As is well known (see e.g. \cite[\S 4]{FM77}, or \cite[Lem. 3b]{CFW81}) the factor space $Y/\cT$ consisting of equivalence classes of $\cT$ is a standard Borel space and admits measurable sections $J_1, \dots,J_N : Y/\cT \to Y$ such that $\set{J_i(\cT(y))\,;\, 1\le i\le N}=[y]_\cT=\cT(y)$ for almost every $y$. Define the cyclic permutation $\sigma_{\cT(y)}$ given by 
$J_1(\cT(y))\mapsto  J_2(\cT(y))\mapsto  \cdots J_N(\cT(y))\mapsto J_1(\cT(y))$ in each class, whose cycle length is $N$. Denote by $\phi_\cT$ the map on $Y$ which coincides with $\sigma_{\cT(y)}$ on each class $\cT(y)$, and denote the cyclic group generated by $\phi_\cT$ by $\Phi_\cT$. Then $\phi_\cT$ is measurable and constitutes an inner automorphism of $\cR$ which leaves invariant almost every class of the relation $\cT$, and the group $\Phi_\cT$ generates the relation $\cT$.  

Applying this procedure to each of the finite bounded relations $\cR_n$ and taking the union of the corresponding groups, the stated result follows. 

\end{proof}

\section{Pointwise covering lemmas} \label{sec:tiling}

We now establish pointwise decomposition results for subset functions which will provide
the central tool for proving the Shannon-McMillan-Breiman theorem in the 
next section. Our main lemma is motivated by the techniques developed by Lindenstrauss
for the proof of the corresponding covering lemma for tempered F{\o}lner 
sequences in amenable groups (see \cite[Lem. 2.4]{Li01}). However, working with hyperfinite 
exhaustions, we are able to avoid two material technicalities: 
\begin{itemize}
\item we do not need to use an auxiliary random parameter in order to be able to choose the desired coverings with high probability;
\item we are able to produce strictly disjoint coverings, and the discussion 
of $\delta$-disjointness (see \cite[Lem. 2.6, 2.7]{Li01}) becomes unnecessary.
\end{itemize}
Instead, we exploit the disjointness properties inherent in a sequence of nested equivalence relations, but nevertheless, the proof of pointwise almost sure convergence is still quite difficult and technical.  

As usual, $\cR$ will denote a p.m.p amenable equivalence relation
over $(Y,\nu)$ with $\tilde{\nu}$ denoting the invariant measure on $\cR$, and 
$(\cR_n)$ will denote a bounded hyperfinite exhaustion for $\cR$.

We start with the following elementary covering (and disjointification) lemma.

\begin{Proposition} \label{prop:basiccov}
Let $N,L \in \NN_{\ge 1}$ with $N < L$ and consider an arbitrary finite sequence of subset functions 
$\cB_j \subseteq \cR_L$, $1 \leq j \leq N$. Further, for $y\in Y$, consider a collection of classes (of the relations $\cR_n$ where $1\le n\le N$) given by 
\[
\mathfrak{F}(y) := \big\{ \cR_{n(j)}(w)\,|\, w \in \cB_j(y), 1 \leq j \leq N \big\}.
\]

Then, for a.e. $y \in Y$, 
we can extract from $\mathfrak{F}(y)$ a disjoint subcollection $\mathfrak{S}(y)$ of classes 
such that 
\begin{eqnarray*}
\coprod_{C \in \mathfrak{S}(y)} C \supset \bigcup_{j=1}^N \cB_j(y)\,, \text{ and so }
\sum_{C \in \mathfrak{S}(y)} \abs{C} \geq \Big| \bigcup_{j=1}^N \cB_j(y) \Big|. 
\end{eqnarray*}
\end{Proposition}

\begin{proof}
Fix $y \in Y$.
We note first that for $z_1, z_2 \in \cR_L(y)$ and $1 \leq j_1, j_2 \leq N$,
there are three possibilities for the inclusion relation between $\cR_{j_1}(z_1)$ 
and $\cR_{j_2}(z_2)$:
\begin{itemize}
\item  $\cR_{j_1}(z_1) \cap \cR_{j_2}(z_2) = \emptyset$,
\item  $\cR_{j_1}(z_1) \subseteq \cR_{j_2}(z_2)$, \quad or
\item  $\cR_{j_1}(z_1) \supseteq \cR_{j_2}(z_2)$.
\end{itemize}
Any collection $\mathfrak{F}(y)$ of sets with the latter property has the property that the union of its constituents has a unique representation as a disjoint union of some of the constituents. 
To find this representation explicitly,  namely to choose the  subcollection
$\mathfrak{S}(y)$, enumerate all the elements in $\mathfrak{F}(y)$ and give them
distinct labels collected in an index set $\cI$. Then run the following checking algorithm.
\begin{enumerate}[(1)]
\item Set $\mathfrak{S}^{*}(y) = \emptyset$, $\cI^{*} = \cI$.
\item Check an arbitrary class $C \in \mathfrak{F}(y)$ 
with its corresponding label being contained in $\cI^{*}$. 
\item Given $C$, there are two possibilities:
\begin{enumerate}[(A)]
\item Either for all $1 \leq j \leq N$ and $z \in \cB_j(y)$ such that 
$\cR_{n(j)}(z)\cap C \neq \emptyset$, we have $C \supseteq \cR_{n(j)}(z)$,  
\item or there is some $1 \leq j_1 \leq N$, $z_1 \in \cB_{j_1}(y)$ such that
$C \subseteq \cR_{n(j_1)}(z_1)$ and $C \neq \cR_{n(j_1)}(z_1)$.
\end{enumerate}
\item Only in case of (A),  add $C$ to the subcollection $\mathfrak{S}^{*}(y)$. 
Then remove from $\cI^{*}$ all labels corresponding to classes
$\cR_{n(j)}(z) \in \mathfrak{F}(y)$ being contained in $C$. If the new set $\cI^{*} = \emptyset$,
 jump to step (5), otherwise  return to step (2).  
\item We have $\cI^{*} = \emptyset$ (meaning all classes have been checked) and we set 
$\mathfrak{S}(y) = \mathfrak{S}^{*}(y)$. This is the collection we aim for. 
\end{enumerate}
By construction, for a.e. $Y$, the elements $C \in \mathfrak{S}(y)$ are pairwise disjoint. Also, we made
sure that for every $1 \leq j \leq N$, every single $b \in \cB_j(y)$ is 
contained in some class $C$ taken into $\mathfrak{S}(y)$. This proves
the above inequality. 
\end{proof}

We now prove the main covering lemma, motivated by 
\cite[Lemma 2.1]{Li01}. 

\begin{Lemma}
 \label{lemma:abstrcomb}
Fix $0 < \delta < 1$ and fix an (arbitrary) finite non-empty set $D\subset \Phi$. 
Then, for sufficiently large $M \in \NN$, depending only on $D$ and $\delta$, the  following property holds. 

Let  ${\cT}_{i,j} \subseteq \cR_L$ 
$(1 \leq i \leq M, 1 \leq j \leq N_i)$ be  
an array of subset functions, such that for a.e. $y$, $\cT_{i,j}(y)=\cR_{n(i,j)}(y)$, where $n(i,j)\le L$. Assume that for $2\le i \le M$ and every $1\le j\le N_i$
\begin{eqnarray} \label{eqn:tempinfiber}
\Big| \bigcup_{k < i} D \circ \big({\cT}^{-1}_{k,*} \,{\cT}_{i,j} \big) \Big| \leq (1+\delta)\,\big|{\cT}_{i,j}\big|
\end{eqnarray}
almost surely, where ${\cT}_{k,*} := \bigcup_{j=1}^{N_k} {\cT}_{k,j}$ for $1 \leq k \leq M$. 

Then, given another array $\cB_{i,j} \subseteq \cR_L$
for $1 \leq i \leq M$ and $1 \leq j \leq N_i$, for $\nu$-almost every $y\in Y$ 
there are disjoint subcollections
\begin{eqnarray*}
\mathfrak{S}(y) \subseteq \big\{ {\cT}_{i,j}(w)\,|\, w \in \cB_{i,j}(y), 1 \leq i \leq M, 1 \leq j \leq N_i \big\}
\end{eqnarray*}
such that 
\begin{eqnarray*}
\sum_{C \in \mathfrak{S}(y)} \abs{C} \geq (1-\delta)\, \min_{1 \leq i \leq M} \Big|D \circ \bigcup_{j=1}^{N_i} \cB_{i,j}(y) \Big|.
\end{eqnarray*}
\end{Lemma}

\begin{proof}
Fix a conull set $Y_0 \subseteq Y$ such that for all $y \in Y_0$ the inequality~\eqref{eqn:tempinfiber}
is fulfilled.  For $i=M$, apply Proposition~\ref{prop:basiccov} to the subset functions defined as  $\cB_j :=\cB_{M,j}$, where $1 \leq j \leq N_M$.
This way, by passing to another conull subset of $Y_0$, for $\nu$-a.e.\@ $y \in Y_0$, we obtain a disjoint subcollection 
\[
\mathfrak{S}_M(y) \subseteq \big\{ {\cT}_{M,j}(w)\,|\, 1 \leq j \leq N_M, w \in \cB_{M,j}(y) \big\}.
\]
with 
\[
\sum_{C \in \mathfrak{S}_M(y)} \abs{C} \geq \Big| \bigcup_{j=1}^{N_M} \cB_{M,j}(y) \Big|.
\]
Proceeding iteratively, for $i < M$, we set
\[
\tilde{\cB}_{i,j}(y) := \cB_{i,j}(y) \setminus \bigcup_{l > i} {\cT}^{-1}_{i,j} \circ (\bigcup \mathfrak{S}_l )(y),
\]
where $1 \leq j \leq N_i$ and $\bigcup \mathfrak{S}_l$ denotes  the union over all sets in $\mathfrak{S}_l$.
Applying Proposition~\ref{prop:basiccov} again, this time to the subset functions defined as  $\cB_j := \tilde{\cB}_{i,j}$, $1 \leq j \leq N_i$
gives a subcollection
\[
\mathfrak{S}_i (y)\subseteq \big\{ {\cT}_{i,j}(w), 1 \leq j \leq N_i, w \in \tilde{\cB}_{i,j}(y) \big\}
\] 
with the corresponding covering property, namely 
\[
\sum_{C \in \mathfrak{S}_i(y)} \abs{C} \geq \Big| \bigcup_{j=1}^{N_i} \tilde{\cB}_{i,j}(y) \Big|.
\]

Having constructed $\mathfrak{S}_i$ for all $1 \leq i \leq M$,
we finally set $\mathfrak{S} := \bigcup_{i=1}^M \mathfrak{S}_i$. By construction, $\mathfrak{S}$ is a disjoint collection of
sets. We are going to show
\begin{eqnarray} \label{eqn:whattoshow}
\sum_{i=m}^M \sum_{C \in \mathfrak{S}_i(y)} \abs{C} \geq \min\big\{1-\delta, (M-m+1) \,\frac{\delta^2}{|D|} \big\}
\cdot \min_{m \leq i \leq M} \Big| D \circ \bigcup_{j=1}^{N_i} \cB_{i,j}(y) \Big|
\end{eqnarray}
for all $1 \leq m \leq M$ and $\nu$-almost every $y\in Y$. 
The proof is by induction on $m=M,\dots, 1$, starting with $M$. 
By the first case in our argument 
above,
\[
\sum_{C \in \mathfrak{S}_M(y)} \abs{C}\geq \Big| \bigcup_{j=1}^{N_M} \cB_{M,j}(y) \Big| 
\geq \frac{1}{|D|} \, 
\Big| D\circ\bigcup_{j=1}^{N_M} \cB_{M,j}(y) \Big|,
\]
which shows the validity of an inequality which is in fact stronger than \ref{eqn:whattoshow}. 
In order to show the claim for $m < M$, we assume that for the element $y \in Y$ under consideration,
\begin{eqnarray} \label{eqn:nothingtoprove}
\sum_{l > m} \sum_{C \in \mathfrak{S}_l(y)} \abs{C} < (1-\delta)\, \min_{m \leq i \leq M} 
\Big| D \circ \bigcup_{j=1}^{N_i} \cB_{i,j}(y) \Big|,
\end{eqnarray}
since, otherwise, there is nothing to prove. By construction of the collection $\mathfrak{S}_m(y)$,
we have
\begin{eqnarray} \label{eqn:covaux11}
\sum_{C \in \mathfrak{S}_m(y)} \abs{C}\geq \Big| \bigcup_{j=1}^{N_m} \tilde{\cB}_{m,j}(y) \Big|.
\end{eqnarray}
It follows from 
\[
\tilde{\cB}_{m,j}(y) = {\cB}_{m,j}(y) \setminus  \bigcup_{l > m} \bigcup_{C \in \mathfrak{S}_l(y)} {\cT}_{m,j}^{-1}\,C
\]
that 
\begin{eqnarray} \label{eqn:covaux22}
\Big(D \circ \bigcup_{j=1}^{N_m} \tilde{\cB}_{m,j}\Big)(y) \supseteq \Big(D\circ \bigcup_{j=1}^{N_m} \cB_{m,j}\Big)(y) 
\setminus \bigcup_{l > m} \bigcup_{C \in\mathfrak{S}_l(y)}
D \circ \Big( \bigcup_{j=1}^{N_m} {\cT}^{-1}_{m,j}\, C \Big)\,.
\end{eqnarray}
Recall that by assumption (\ref{eqn:tempinfiber}), we have
\[
\Big| D \circ \Big( \bigcup_{j=1}^{N_m} \cT_{m,j}^{-1}\,C  \Big) \Big| \leq (1+\delta)\, \big| C \big|
\]
for a.e.\@ $y \in Y_0$, where $C \in \bigcup_{l > m} \mathfrak{S}_l(y)$. Hence, by \eqref{eqn:covaux22},
\begin{eqnarray*}
\Big| \bigcup_{j=1}^{N_m} \tilde{\cB}_{m,j}(y) \Big| &\geq& \frac{1}{|D|}\, \Big| \Big(D \circ \bigcup_{j=1}^{N_m} \tilde{\cB}_{m,j}\Big)(y) \Big| \\
&\geq& \frac{1}{|D|} \, \Big( \Big| D \circ \bigcup_{j=1}^{N_m} \cB_{m,j}(y) \Big| 
-(1+\delta) \sum_{l > m} \sum_{C \in\mathfrak{S}_l} \abs{C} \Big). 
\end{eqnarray*}
With \eqref{eqn:covaux11} we obtain by taking into account \eqref{eqn:nothingtoprove}
\begin{eqnarray*}
\sum_{C \in \mathfrak{S}_m} \abs{C} &\geq&  \frac{1}{|D|} \, \Big| D \circ \bigcup_{j=1}^{N_m} \cB_{m,j}(y) \Big| - 
\frac{1+\delta}{|D|}\, \sum_{l > m} \sum_{C \in\mathfrak{S}_l} \abs{C}\\
&\geq& \big( 1- (1+\delta)(1-\delta) \big)\, \frac{1}{|D|}\, \min_{m \leq i \leq M} \Big|  D \circ \bigcup_{j=1}^{N_i} \cB_{i,j}(y) \Big| \\
&=& \frac{\delta^2}{|D|}\, \min_{m \leq i \leq M} \Big| D \circ \bigcup_{j=1}^{N_i} \cB_{i,j}(y) \Big|.
\end{eqnarray*}
Consequently,
\begin{eqnarray*}
\sum_{l \geq m} \sum_{C \in \mathfrak{S}_l} \abs{C} &\geq& \sum_{l > m} \sum_{C \in \mathfrak{S}_l} \abs{C}  + 
\frac{\delta^2}{|D|}\, \min_{m \leq i \leq M} \Big| D \circ \bigcup_{j=1}^{N_i} \cB_{i,j}(y) \Big| \\
&\geq& \Big( \big(M - (m+1) + 1 \big) \, \frac{\delta^2}{|D|} + \frac{\delta^2}{|D|}  \Big) \, \min_{m \leq i \leq M}\,\Big| D \circ \bigcup_{j=1}^{N_i} \cB_{i,j}(y) \Big|.
\end{eqnarray*} 
Hence, the proof of the inequality~\eqref{eqn:whattoshow} is complete. 
Finally, choosing $M$ large enough such that 
\[
\frac{(M-1)\delta^2}{|D|} \geq (1-\delta)
\]
concludes the proof of the lemma. 
\end{proof}


\section{Proof of the Shannon-McMillan-Breiman theorem} \label{sec:SMB}

The goal of this section is to prove Theorem~\ref{thm:MAIN_SMBhyp}, as
well as~Corollary~\ref{cor:MAIN_SML1}. To this end, we adapt the overall strategy given 
in \cite[Section 4]{Li01} to the situation of amenable equivalence relations.
As usual, $\cR$ is p.m.p.\@ and hyperfinite, $\Gamma$ is countable and  $\alpha: \cR \to \Gamma$ a  class injective  measurable cocycle. Recall that for a
p.m.p.\@ group action $\Gamma \curvearrowright (X,\lambda)$, 
the extended equivalence relation $\cR^X$ over
$\big( X \times Y, \lambda \times \nu \big)$ is defined by the condition 
\[
\big( (x,y), (x^\prime,y^\prime) \big) \in \cR^X \Longleftrightarrow y\cR y^\prime \,\,\text{ and } x=\alpha(y,y^\prime)x^\prime
\,,\]
%
When the measure $\nu$ is $\cR$-invariant, it follows that $ \lambda \times \nu$ is $\cR^X$-invariant,
since the $\Gamma$-action on $X$ preserves $\lambda$. The projection map $\pi : \cR^X\to \cR$ given by $(x,y)\to y$ is {\it injective} when restricted to the $\cR^X$-equivalence class of $(x,y)$, for almost all $(x,y)\in \cR^X$.  Further, it 
is well known that an extension of an amenable action is amenable, and thus in particular if $\cR$ is amenable, so is $\cR^X$. But since the extension is class-injective, in fact every hyperfinite exhaustion $(\cR_n)$ of $\cR$ can be canonically lifted to a hyperfinite exhaustion $(\cR_n^X)$ of $\cR^X$, via $\cR_n^X((x,y))=\set{(\alpha(z,y)x,z)\,;\, z\in \cR_n(y)}$. Note that if $(\cR_n)$ is a bounded hyperfinite exhuastion, then so is $(\cR_n^X)$, with the same bounds on the equivalence classes.

For a given hyperfinite exhaustion $(\cR_n)$ for $\cR$, we define a set 
$\Phi \subseteq \operatorname{Aut}(\cR)$ satisfying the conclusions of Proposition~\ref{prop:finitephin}.
Then, every $\phi \in \Phi$ can be extended naturally to an inner automorphism
$\phi^X \in \operatorname{Aut}(\cR^X)$ by setting
\begin{equation}\label{phi^X}
\phi^{X} \big( (x,y) \big) := \big( \alpha(\phi(y),y)x, \phi(y) \big).
\end{equation}
For a subset $D \subseteq \Phi$, we write $D^X$ for the set 
$\big\{ \phi^X\,|\, \phi\in D \big\}$. Note that by definition, the set 
$\Phi^X:= \{ \phi^X\,|\, \phi \in \Phi\}$ is generating for the relation $\cR^X$.
Clearly, if $\phi$ preserves the classes of $\cR_n$ almost surely, then $\phi^X$ preserves the classes of $\cR_n^X$ almost surely.

Let us now recall the concept of ergodicity for a
measure preserving equivalence relation. Let $\cZ$ be such a relation over a probability space $(Z,\eta)$. 
A subset
$A \subseteq Z$ will be called $\cZ$-{\em invariant} if 
\[
\big( A \times Z \big)\, \cap \cZ =( A \times A)\cap \cZ. 
\]
The relation $\cZ$ is {\em ergodic} if every 
$\cZ$-invariant set $A \subseteq Z$ satisfies $\eta(A) \in \{0,1\}$.
From now on, we will always assume that the relation $\cR^X$ is {\em ergodic}.
For sufficient conditions guaranteeing this property, we refer to the discussion in \S \ref{sec:ergodicity} below.



One crucial ingredient for the proofs of Theorem~\ref{thm:MAIN_SMBhyp} and Corollary~\ref{cor:MAIN_SML1}
is the pointwise ergodic theorem. A general form of it being valid for suitable asymptotically invariant sequences of subset functions in an  amenable equivalence relation
was established in \cite[Thm.\@ 2.1]{BN13b}. 
Since we restrict our discussion here to hyperfinite exhaustions, we will state a less general but easily accessible special case which is sufficient for our purposes. Indeed, the following fact is an immediate consequence of the martingale convergence theorem.  

\begin{Theorem} \label{thm:ETequiv}
Let $\cZ$ be an ergodic, p.m.p.\@  equivalence relation 
over a probability space $(Z,\eta)$. Let $(\cZ_n)$ be a bounded hyperfinite exhaustion for $\cZ$. \\ 
Then for all $f \in L^1(Z,\eta)$, we have 
\begin{eqnarray*}
\lim_{n \to \infty} \big| \cZ_n(z) \big|^{-1}\,\sum_{w \in \cZ_n(z)} f(w)
 = \int_Z f(w)\, d\eta(w)
\end{eqnarray*}
for $\eta$-almost every $z \in Z$. 
\end{Theorem}



We now turn to establish some preliminary 
lemmas that will be used below. 
The first step is a straightforward consequence of the ergodicity of
the relation. 

\begin{Lemma} \label{lemma:ergodicity}
Let $\cZ$ be an ergodic p.m.p. equivalence relation over a probability $(Z, \eta)$ along
with a countable set $\Phi$ of inner automorphisms generating $\cZ$.
Then, for every $\delta > 0$, and every set $A \subseteq Z$ with 
$\eta(A) > 0$, there is a finite set $D \subset \Phi$ such that 
\[
\eta\big( D \circ A \big) \geq 1 - \delta. 
\]
\end{Lemma}
\begin{proof}
Let $A \subseteq Z$ be measurable with $\eta(A) > 0$. 
Assume that there is some $\delta_0 > 0$ such that 
for all finite collections $D \subset \Phi$, we have $\eta(D \circ A) < 1 - \delta_0$.
We define $\overline{A} := \Phi \circ A = \bigcup_{\phi \in \Phi} \phi(A)$. This 
set is invariant under the relation $\cZ$. To see this, consider $z \in Z$, $\phi \in \Phi$ and $a \in A$ such that 
$(\phi(a), z) \in \cZ$. Since $\Phi$ is generating, there is a further
inner automorphism $\phi^{\prime} \in \Phi$ such that $z = \phi^{\prime}\phi(a)$.
We conclude that $z \in \Phi \circ A$ and thus, 
\[
\big( \overline{A} \times Z \big) \cap \cZ = (\overline{A} \times \overline{A})\cap \cZ.
\]  
Since $\Phi$ is countable, it follows $\eta(\overline{A}) \le 1 -\delta_0$ by
our assumption. Now $\overline{A}$ is an invariant set, hence by ergodicity
of the relation, we obtain $\eta(\overline{A}) = 0$. However, since every
inner automorphism preserves the measure, we have $\eta(\overline{A}) \geq
\eta(A) > 0$. This is a contradiction. 
\end{proof}

We need the following combinatorial covering lemma, which establishes  the measured equivalence relation
version of a covering lemma originally formulated for F{\o}lner sequences, see \cite[Lem.\@ 1]{OW83}
and \cite[Lem.\@ 4.2]{Li01}.
\begin{Lemma} \label{lemma:comb}
Let $\cR$ be a p.m.p. equivalence relation over $(Y,\nu)$, and let $(\cR_n)$ be a bounded hyperfinite exhaustion satisfying  
$\lim_{n \to \infty} \operatorname{ess}\,\operatorname{inf}_y |\cR_n(y)| = \infty $.
Then, for every $\eta > 0$, there is some $\ell \in \NN$ such that the following holds.

Suppose that to $y \in Y$  
there corresponds a finite increasing sequence $(k_i), 1 \leq i \leq r(y)$ of integers (depending on $y$)  with  $|\cR_{k_1}(z)| \geq \ell$
for almost every $z \in Y$. 
Then, there is $n_0 \in \NN$ such that for all $n \geq n_0$, the number of 
possible disjoint subcollections $\mathfrak{S}(y)$ of the form 
\[
\mathfrak{S}(y) \subseteq \big\{{\cR}_{k_i}(c) \,|\, c \in \cR_n(y), 1 \leq i \leq r(y) \big\}
\]
is at most $2^{\eta|\cR_n(y)|}$.
\end{Lemma}

\begin{proof}
Let $\eta > 0$. Consider $\ell \in \NN$, $y \in Y$ and
an increasing  sequence $(k_i), 1 \leq i \leq r(y)$ of integers (depending on $y$) with $|\cR_{k_1}(z)| \geq \ell$ 
for all $1 \leq i \leq r(y)$
and almost every $z \in Y$. 
Since we assume that $\mathfrak{S}(y)$ is a collection of disjoint ${\cR}_{k_i}$-classes contained
in $\cR_n(y)$, there is a center set $\cC(y) \subseteq \cR_n(y)$ such that 
\[
\mathfrak{S}(y) = \big\{ {\cR}_{k_{t(c)}}(c)\,|\, c \in \cC(y)\big\} \text{  with  }  1 \leq t(c) \leq r(y) .
\]
Given $\mathfrak{S}(y)$ (and hence $\cC(y)$), we define a set of points $\cN(y)$ as follows. 
For $c \in \cC(y)$, define $n(c):= i$ as the maximal $1 \leq i \leq t(c)$ such that
$\cR_{k_i}(c) \setminus \cR_{k_{i-1}}(c) \neq \emptyset$, where $\cR_0 = \emptyset$ by 
convention. Then add to the set $\cN(y)$ an arbitrary point $p(c) \in \cR_{k_i}(c) \setminus \cR_{k_{i-1}}(c)$.
By processing in
that way for all $c \in \cC(y)$, we obtain a set $\cN(y)$ with cardinality at most $\abs{\cC(y)}$. We claim that   
we can uniquely recover $\mathfrak{S}(y)$ from knowing the elements in both sets $\cC(y)$ and $\cN(y)$.
Indeed, given $\cC(y)$ and $\cN(y)$, by construction, for every $c \in \cC(y)$, 
there is a minimal (and thus unique) $1 \leq t(c) \leq r(y)$ such that ${\cR}_{k_{t(c)}}(c) \cap \cN(y) \neq \emptyset$. 
Hence, 
the number of possible disjoint subcollections must be bounded by the number of choices for
the two sets $\cC(y)$ and $\cN(y)$. 
Since the
sizes of all classes involved are uniformly bounded from below by $\ell$, the cardinality of 
every $\cC(y)$ and $\cN(y)$ is bounded from above by $\lceil |\cR_n(y)|/\ell \rceil$.
In light of that, we need to bound the expression
\begin{eqnarray*}
\big( \lceil|\cR_n(y)|/\ell\rceil + 1 \big)^2 \, \binom{|\cR_n(y)|}{\lceil |\cR_n(y)| /\ell \rceil}^2.
\end{eqnarray*}
To do so, we use the entropy formula for the Stirling approximation. 
For this purpose, we define 
 \[
E(\ell):= \frac{1}{\ell}\, \log\,\ell + \Big(1 - \frac{1}{\ell} \Big)\,\log\,\Big(1 - \frac{1}{\ell} \Big)^{-1}.  
\]
By Stirling's approximation (see e.g.\@ \cite{FS08}, Example~VIII.10), for large enough $n \in \NN$, we get 
\begin{eqnarray*}
\binom{|\cR_n(y)|}{\lceil|\cR_n(y)|/\ell\rceil}^2 \leq \left(\exp \Big( E(\ell/2)|\cR_n(y)| \Big)\right)^2 \leq 2^{4\,E(\ell/2)\,|\cR_n(y)|}. 
\end{eqnarray*}  

Now increasing $n$ if necessary, we can make sure that
\[
\big( \lceil\cR_n(y)/\ell\rceil +1 \big)^2\, \binom{|\cR_n(y)|}{\lceil|\cR_n(y)|/\ell\rceil}^2 
\leq 2^{5\,E(\ell/2)\,|\cR_n(y)|}.
\]
Since $E(\ell) \to 0$
as $\ell \to \infty$, we can find some $\ell$ such that $E(\ell/2) < \eta/5$, and  
this completes the proof of the claim. 
\end{proof}

We are ready to prove the main lemma of this section. 
It is motivated by Lemma~4.3 in \cite{Li01} and 
provides an analog of it for hyperfinite exhaustions. 

\begin{Lemma}
 \label{lemma:dyncov1}
Let $(B_k)$ be a sequence of measurable sets in $X \times Y$ such that 
\begin{eqnarray*}
\lambda \times \nu \Bigg( \bigcap_{k=1}^{\infty} \bigcup_{j \geq k} B_j \Bigg) > 0. 
\end{eqnarray*}
Then, for every $\delta > 0$ and $\lambda \times \nu$-a.e.\@ $(x,y) \in X \times Y$,
there is $n(x,y) \in \NN$ for which the following holds true: for each $n \geq n(x,y)$,
there is a disjoint collection of subsets of $\cR_n(y)$ 
\[
\mathfrak{S} = \big\{ \cR_{k_i}(b_i)\,|\, 1 \leq i \leq r \big\}
\]
with $b_i \in \cR_n(y)$ and $1 \leq k_i < n(x,y)$, such that 
\begin{enumerate}[(i)]
\item $\big( \alpha(b_i, y)x, b_i \big) \in B_{k_i}$ for all $i$,
\item $\sum_{C \in \mathfrak{S}} \abs{C} \geq (1-\delta)\,|\cR_n(y)|$.
\end{enumerate}
\end{Lemma}

\begin{proof}
Let $\delta > 0$. We set 
\[
B^{*} := \bigcap_{k=1}^{\infty} \bigcup_{j \geq k} B_j.
\]
By assumption, $(\lambda \times \nu)(B^{*})> 0$. 
Given the hyperfinite exhaustion $(\cR_n)$, we fix  
$\Phi^{X} \subseteq \operatorname{Aut}(\cR^X)$ satisfying the conclusions of Proposition~\ref{prop:finitephin}.
Since the extended
relation $\cR^X$ is ergodic, and since the set $\Phi^X$
is generating for $\cR^X$,
 we can use Lemma~\ref{lemma:ergodicity}
in order to find a finite set $D\subset \Phi$ such that the lifted automorphisms $D^X \subseteq \Phi^{X}$
 satisfy 
\begin{eqnarray} \label{eqn:useergodicity}
\lambda \times \nu\big( D^X \circ B^{*} \big) \geq 1- \delta/10.
\end{eqnarray}
Thus clearly $\lambda \times \nu\big( B^{*} \big) \geq (1- \delta/10)/\abs{D}$.
Further,
choose $M \in \NN$ such that Lemma~\ref{lemma:abstrcomb} holds
true for the fixed parameters $\delta$ and $D$ for the relation $\cR$. 

 The next step is
to construct two finite increasing integer sequences $(m_i), (N_i)$, $1 \leq i \leq M$,
according to the following algorithm.
\begin{enumerate}[(1)]
\item Define $m_1:= 1$.  
\item If $m_i$ has been chosen, then determine $N_i$ large enough such that 
\[
\lambda \times \nu \Big( B^{*} \setminus \bigcup_{j=m_i}^{m_i + N_i} B_j \Big)
< \frac{\delta\cdot\lambda\times \nu(B^\ast)}{10\,M\,|D|}.
\]
\item Further, if $N_i$ has been chosen, choose $m_{i+1}$ large enough such that for
every $l \geq m_{i+1}$, we get
\[
\bigcup_{j < m_i + N_i} D \circ \cR_j^{-1}\cR_{l} \subseteq \cR_l.
\]
This is possible by Proposition~\ref{prop:finitephin} and since $D\subset \Phi$.  
In fact $D\circ \cR_n =\cR_n$ for all sufficiently large $n$. So it suffices that $m_{i+1} > m_i+N_i$ and also that $\cR_{m_{2}}$ (and hence each $\cR_{m_{i+1}}$) is invariant under $D$.

\end{enumerate}
With the sequences $(m_i), (N_i)$ at our disposal, we define
\[
\tilde{B}^{*} := \bigcap_{i=1}^M \bigcup_{j=m_i}^{m_i + N_i} B_j. 
\]
Clearly 
\[
\lambda \times \nu \big(   B^{*} \setminus  \tilde{B}^{*} \big) \leq
\sum_{i=1}^M \lambda \times \nu \Big( B^{*} \setminus \bigcup_{j=m_i}^{m_i+ N_i} B_j \Big)
< \frac{\delta\cdot\lambda\times \nu(B^\ast)}{10\,|D|},
\]
and so 
\[
\lambda \times \nu \big( D^X \circ  B^{*} \setminus D^X \circ \tilde{B}^{*} \big) <
\frac{\delta\cdot\lambda\times \nu(B^\ast)}{10}.
\]
Consequently, 
\[
\lambda \times \nu\big( D^X \circ \tilde{B}^{*} \big) \geq \nu\big( D^X \circ {B}^{*} \big) - \frac{\delta\cdot\lambda\times \nu(B^\ast)}{10}
\geq 1-\frac{\delta}{10}- \frac{\delta\cdot\lambda\times \nu(B^\ast)}{10}\ge 1-\frac{\delta}{5}.   
\]
Now by the pointwise ergodic theorem (Theorem~\ref{thm:ETequiv}), for $\lambda \times \nu$-a.e.\@ 
$(x,y)$, we can define $n(x,y)$ as the smallest integer value greater than $m_M + N_M$
 such that for all $n \geq n(x,y)$,
we have 
\begin{eqnarray} \label{eqn:ErgThm}
\cA\big( n, \one_{D^X\circ \tilde{B}^{*}} \big)(x,y):= 
\big| \cR_n(y) \big|^{-1} \sum_{z \in \cR_n(y)} \one_{D^X \circ \tilde{B}^{*}}\big( (\alpha(z,y)x,z) \big) > 1 - \frac{\delta}{4}.
\end{eqnarray} 
We fix some pair $(x,y)$ satisfying this condition, as well as $n \geq n(x,y)$. For $1 \leq i \leq M$ and $1 \leq j \leq N_i$, set
${\cT}_{i,j}:= \cR_{m_i + j - 1}$ and 
\[
A_{i,j}(y) := \big\{ b \in \cR_n(y) \,|\, \big( \alpha(b,y)x, b \big) \in B_{m_i + j -1} \big\}.
\]
By step~(3) of the algorithm and since $n > m_M + N_M$, 
we have $D \circ \cR_n \subseteq \cR_n$. In particular, each $\phi\in D$ gives rise to bijections $\phi :\cR_n(y)\to \cR_n(y)$, for a.e. $y\in Y$.

Applying the transformations in the set $D$ to the sets $A_{i,j}(y)$, we then obtain 
\begin{eqnarray*}
\end{eqnarray*}
\begin{eqnarray*}
D \circ \bigcup_{j=1}^{N_i} A_{i,j}(y) &=& \bigcup_{\phi \in D} \big\{ \phi(b) \,|\, b \in \cR_n(y), \big( \alpha(b,y)x, b \big) 
\in \bigcup_{j=1}^{N_i} B_{m_i + j -1} \big\} \\
&=& \big\{ b \in \cR_n(y)\,|\,\exists\, \phi \in D,\, b^{\prime} \in \cR_n(y): \, b = \phi(b^{\prime}),\, \big( \alpha(b^{\prime}, y)x, b^{\prime} \big)
\in \bigcup_{j=1}^{N_i} B_{m_i + j -1} \big\} .
\end{eqnarray*}
Together with~\eqref{eqn:ErgThm}, this yields for all $1 \leq i \leq M$ :
\begin{eqnarray*}
\Big| \bigcup_{j=1}^{N_i} D \circ A_{i,j}(y) \Big| &=&  \big| \big\{ b\in \cR_n(y)\,|\, \exists\, \phi \in D:\,
 b^\prime = \phi^{-1}(b), \, \big( \alpha(b^{\prime}, y)x, b^{\prime} \big)
\in \bigcup_{j=1}^{N_i} B_{m_i + j -1}  \big\}\big|   \\ \text{(using (\ref{phi^X}))}
&=& \big| \big\{ b \in \cR_n(y)\,|\, \big( \alpha(b,y)x, b \big) \in D^X \circ \bigcup_{j=1}^{N_i} B_{m_i + j -1} \big\} \big| \\\text{(since $\tilde{B}^\ast$ is an intersection)}
&\geq& \big| \cR_n(y) \big|\, \cA\big(n, \one_{D^X \circ \tilde{B}^{*}}  \big)(x,y) \\ 
(\text{by the pointwise ergodic theorem}) 
&\geq&\Big( 1- \frac{\delta}{4} \Big) \,\big| \cR_n(y) \big|\,.
\end{eqnarray*}

We finally apply Lemma~\ref{lemma:abstrcomb} with $\delta/4$ instead of $\delta$ to the arrays ${\cT}_{i,j}$ and $A_{i,j}(y)$, where $1 \leq i \leq M$ and
$1 \leq j \leq N_i$. The assumption of the Lemma  is indeed satisfied, namely $\cT_{i,j}=\cR_{m_i+j-1}$ satisfy (5.1) by the construction of $m_2$, which guarantees $D\circ \cR_m=\cR_m$ for all $m \ge m_2$. Furthermore, we have just shown that 
$$\min_{1\le i \le M}\Big| \bigcup_{j=1}^{N_i} D \circ A_{i,j}(y) \Big|\ge \Big( 1- \frac{\delta}{4} \Big) \,\big| \cR_n(y) \big|$$
and hence 
Lemma~\ref{lemma:abstrcomb} implies that there is a disjoint subcollection
\[
\mathfrak{S}(y) \subseteq \big\{\cT_{i,j}(b) \,|\, b \in A_{i,j}(y) \big\}
\]
with 
\[
\sum_{C \in \mathfrak{S}(y)} \abs{C} \geq (1-\delta)\,\big| \cR_n(y) \big|,
\]
as desired. 
\end{proof}

The following is an immediate consequence from the previous 
lemma. 

\begin{Lemma}
 \label{lemma:dyncov2}
Let $(\cR_n)$ be a bounded hyperfinite exhaustion, and keep the assumptions of the previous lemma. Then, for the number
\begin{eqnarray*}
h:= \operatorname{ess-inf}_{x,y} \, \liminf_{n \to \infty} \frac{\cJ\big( \cP^{\cR_n(y)}(x)\big)}{|\cR_n(y)|},
\end{eqnarray*}
the following holds true. 
For every $\delta > 0$, each $N \in \NN$ and $\lambda \times \nu$-a.e.\@ $(x,y)$,
there is a number $n(x,y) \in \NN$ such that for all $n \geq n(x,y)$, we can find
a disjoint collection (depending on both $x$ and $y$)
\[
\mathfrak{S} = \big\{ \cR_{k_i}(b_i) \,|\, 1 \leq i \leq r \big\}
\]
such that all $b_i \in \cR_n(y)$ and $N \leq k_i < n(x,y)$ and further,
\begin{enumerate}[(i)]
\item for all $i$, 
\[
\frac{\cJ\big( \cP^{\cR_{k_i}(b_i)}\big( \alpha(b_i,y)x\big) \big)}{|\cR_{k_i}(b_i)|} \leq h + \delta, 
\]
\item  $\sum_{C \in \mathfrak{S}} \abs{C}\geq (1-\delta)\,|\cR_n(y)|$.  
\end{enumerate}
\end{Lemma}

\begin{proof}
We apply the previous Lemma~\ref{lemma:dyncov1} to the sequence 
$(\cR_n)_{n \geq N}$ and the sets 
\[
B_k := \Big\{ (x,y) \in X \times Y\, \big| \, \frac{\cJ\big( \cP^{\cR_{N+k}(y)}(x) \big)}
{|\cR_{N + k}(y)|} \leq h + \delta\,\, 
 \Big\}.
\]
\end{proof}

We are now in position to prove the Shannon-McMillan-Breiman Theorem. 
To do so, we combine the previous results of this section, motivated by  
the proof of Theorem~1.3 in~\cite{Li01}.

\begin{proof}[{Proof of Theorem~\ref{thm:MAIN_SMBhyp}}]
As before, we set
\begin{eqnarray*}
h:= \operatorname{essinf}_{x,y} \, \liminf_{n \to \infty} \frac{\cJ\big( \cP^{\cR_n(y)}(x)\big)}{|\cR_n(y)|}.
\end{eqnarray*}
If $h= \infty$, then there is nothing left to show. So assume that $h < \infty$. 
We show that for a.e.\@ $(x,y)$,
\begin{eqnarray} \label{eqn:mainclaim}
\limsup_{n \to \infty} \frac{\cJ\big( \cP^{\cR_n(y)}(x) \big)}{|\cR_n(y)|} \leq h. 
\end{eqnarray}
To this end, fix $\delta > 0$. 
We find $N \in \NN$ large enough such that 
Lemma~\ref{lemma:comb} holds for $\eta = \delta$ and $\ell = N$.
By the growth assumption on the $(\cR_n)$, we find $N_1 \in \NN$ such that for 
a full measure set of $y$'s, we have $|\cR_k(z)| \geq N$ for all $k \geq N_1$ and
each $z \in [y]$.   
We have seen in Lemma~\ref{lemma:dyncov2} that for 
$\lambda \times \nu$-almost every $(x,y)$, there is $n(x,y) \in \NN$, $n(x,y) > N_1$,
such that for $n > n(x,y)$, we obtain a special subcollection 
$\mathfrak{S}=\mathfrak{S}(x,y)$ of
subsets of $\cR_n(y)$. Namely, this collection 
\begin{itemize}
\item is disjoint, 
\item satisfies $\sum_{C \in \mathfrak{S}}\abs{C} \geq (1-\delta)\,|\cR_n(y)|$,
\item Each $C \in \mathfrak{S}$ is of the form $\cR_{k}(b)$ with $N_1 \leq k< n(x,y)$ and $b \in \cR_n(y)$, which 
by the choice of $N_1$ implies that $|\cR_k(b)| \geq N$ for all $k \geq N_1$, 
\item 
writing $C= \cR_k(b) \in \mathfrak{S}$, the following holds
\begin{eqnarray} \label{eqn:infaux}
\frac{\cJ\big( \cP^{\cR_k(b)}\big(\alpha(b,y)x\big) \big)}{|\cR_k(b)|} \leq h + \delta\,.
\end{eqnarray}
\end{itemize}
By increasing $n$ if necessary, due to $|\cR_n(y)| \to \infty$, we can assume that 
\begin{eqnarray} 
\big| \cR_{n}(y) \big| \geq  8\delta^{-1}\,( \log\, n ). \label{eqn:growth} 
\end{eqnarray}
Let us now fix $x$, $y$, $n=n(x,y)$ and $\mathfrak{S}(x,y)=\mathfrak{S}$ satisfying all the conditions stated above. 

We first note that since $\cR_n(y)$ is the disjoint union of the sets $C = \cR_k(b_C) \in \mathfrak{S}$ and  
$\cG_n(y):= \cR_n(y) \setminus \bigcup_{C \in \mathfrak{S}} \cR_k(b_C)$, it follows that the partition 
$\cP^{\cR_n(y)}$ is given by $\left(\bigvee_{C \in \mathfrak{S}} \cP^{\cR_k(b_C)}\right)\vee \cP^{\cG_n(y)}$. 
Therefore for any point $x\in  X$, its $\cP^{\cR_n(y)}$-name arises as the intersection of  
the $\cP^{\cR_k(b_C)}$-name of $\alpha(b_C,y)x$ where $C = \cR_k(b_C) \in \mathfrak{S}$,  and of the 
$\cP$-name of $\alpha(b,y)x$ for  every 
$b \in \cG_n(y):= \cR_n(y) \setminus \bigcup_{C \in \mathfrak{S}} \cR_k(b_C)$.



Consider the partition $\cP^{\cR_n(y)}$ (which is finer than each partition $\cP^{\cR_k(b)}$ when $b\in \cR_n(y)$) and define a set of atoms in it which we denote by $K_n(x,y)$.
Namely, for our fixed $x$, we consider the disjoint collection $\mathfrak{S}(x,y)$ of subsets of $\cR_n(y)$, and we put an atom of $\cP^{\cR_n(y)}$ 
 in $K_n(x,y)$ provided that it arises as the intersection of elements of the partitions $\cP^{\cR_k(b)}$ which satisfy inequality~\eqref{eqn:infaux} (where we allow any $\cR_k(b)=C \in \mathfrak{S}(x,y)$), with elements in the partition $\cP^{\cG_n(y)}$.

Now, for every $\cR_k(b)=C \in \mathfrak{S}(x,y)$,  
inequality \eqref{eqn:infaux} gives a lower bound on the measure of some of the atoms in the partition, and hence an upper bound on their number. It follows that for each such $C$,
there are at most $2^{(h + \delta)|\cR_k(b)|}$ atoms of $\cP^{\cR_k(b)}$ (namely $\cP^{\cR_k(b)}$-names) for which the inequality~\eqref{eqn:infaux}
can hold true.

Consequently, 
the number of atoms of $\cP^{\cR_n(y)}$
which appear as elements in $K_n(x,y)$ 
 is bounded by 
\[
\prod_{C \in \mathfrak{S}} 2^{(h+\delta)\,|\cR_k(b)|} \cdot |\cP|^{|\cG_n(y)|} =
\prod_{C \in \mathfrak{S}} 2^{(h + \delta)\,|C|} \cdot 
\abs{\cP}^{ |\cR_n(y) \setminus \bigcup \mathfrak{S}|}.
\]
Since $\cR_n(y)$ is disjointly $(1-\delta)$-covered by the elements in $\mathfrak{S}$, 
we conclude that for a fixed $x \in X$, 
there are at most 
\[
2^{(h+\delta) \,|\cR_n(y)|} \cdot |\cP|^{\delta\,|\cR_n(y)|}
\]
many elements in $K_n(x,y)$. 

$K_n(x,y)$ depends on both $x$ and $y$, having been constructed using the collection $\mathfrak{S}(x,y)$ of subsets on $\cR_n(y)$. Now, we define another collection of atoms of $\cP^{\cR_m(y)}$, which we denote by $K_m(y)$. It  consists of all the atoms in the sets $ K_m(x,y)$ as $x$ varies in $X$, provided that $n(x,y)$ satisfies the conditions stated above and in addition $n(x,y)\le m$.  By Lemma~\ref{lemma:comb} (and the choices for $N$ and $N_1$), for all $m\ge m(y)$ the number of possibilities for $\mathfrak{S}(x,y)$ is bounded 
by $2^{\delta|\cR_m(y)|}$. 
Hence, we obtain 
\[
\big| K_m(y) \big| \leq 2^{( h + 2\delta + \delta\,\log|\cP| )\,|\cR_m(y)|}.
\] 
We now consider the sets
\[
X_m(y) := \Big\{ x \in X\,\big| \, \frac{\cJ\big( \cP^{\cR_m(y)}(x)\big)}{|\cR_m(y)|} > h + 3 \delta + \delta\,\log|\cP|\Big\}.
\]

Consider $\bigcup K_m(y)$, the union of all $\cP^{\cR_m(y)}$-atoms which belong to $K_m(y)$. 
Then,
\[
\lambda\big( X_m(y) \cap \bigcup K_m(y) \big) \leq \abs{K_m(y)}\,2^{-(h + 3\delta + 
\delta\,\log|\cP|)\,|\cR_m(y)|}
\leq 2^{-\delta\,|\cR_m(y)|}.
\]
It follows from the growth condition~\eqref{eqn:growth} stated above that
\[
2^{-\delta\,|\cR_m(y)|} \leq 2^{\log\, m^{-8}} \leq \frac{1}{m^2}
\]
for large enough $m$. This implies that $\sum_{m=1}^{\infty} 2^{-\delta|\cR_m(y)|} < \infty$. 
Thus, for $\nu$-almost-every $y \in Y$, we can apply the Borel-Cantelli lemma and
obtain that for $\lambda \times \nu$-almost every $(x,y)$, 
$x \notin X_m(y) \cap \bigcup K_m(y) $ if $m$ is large enough. On the other hand, we deduce from
Lemma~\ref{lemma:dyncov2} that for large enough $m$ (depending on $x$ and $y$), 
$x \in \bigcup K_m(y)$. This implies that for $\lambda \times \nu$-a.e. $(x,y)$,
and $m$ large enough, we must have $x \notin X_m(y)$, which means by definition of $X_m(y)$ that
\[
\limsup_{m \to \infty} \frac{\cJ\big( \cP^{\cR_m(y)}(x) \big)}{|\cR_m(y)|} \leq h + 3\delta + \delta\,\log|\cP|. 
\]
Letting $\delta \to 0$ yields \eqref{eqn:mainclaim}. This  establishes almost everywhere
convergence, as stated. Since the integrals of $\cJ\big( \cP^{\cR_m(y)}(x)\big)/\abs{\cR_m(y)}$ over $X\times Y$ converge to the orbital entropy of $\cP$ by Theorem~\ref{thm:MAINcocycleentropy}, we have $h = h^{*}_{\cP}(\alpha)$, as claimed.
\end{proof}

We conclude this section with the proof of Corollary~\ref{cor:MAIN_SML1}.
The proof follows along the lines of the $L^1$-convergence case proved  in
\cite[Thm. 4.1]{Li01}.

\begin{proof}[Proof of Corollary~\ref{cor:MAIN_SML1}]
By Theorem~\ref{thm:MAIN_SMBhyp}, the normalized information function converges 
pointwise almost surely (w.r.t.\@ the measure $\lambda \times \nu$) 
to $h^{*}_{\cP} = h^{*}_{\cP}(\alpha)$. Fix $\varepsilon > 0$, and  
for $n \in \NN$ and fixed $y \in Y$, define 
\begin{eqnarray*}
\mathcal{C}_n(y) := \Big\{ x \,\big| \, h^{*}_{\cP} -
\varepsilon \leq \frac{\cJ\big( \cP^{\cR_n(y)}(x) \big)}{\big| \cR_n(y) \big|} 
\leq h^{*}_{\cP} + \varepsilon \Big\}.
\end{eqnarray*}
Integration over $X$ gives 
\begin{eqnarray*}
\big( h^{*}_{\cP} - \varepsilon \big)\,\lambda\big( \cC_n(y) \big) 
&\leq& \int_{\cC_n(y)} \frac{\cJ\big( \cP^{\cR_n(y)}(x) \big)}{|\cR_n(y)|}\,d\lambda(x) \\
&\leq& \int_X \frac{\cJ\big( \cP^{\cR_n(y)}(x) \big)}{|\cR_n(y)|}\,d\lambda(x) \\ 
&\leq& \big( h^{*}_{\cP} + \varepsilon \big)\,\lambda\big( \cC_n(y) \big) + 
\int_{X \setminus \cC_n(y)} \frac{\cJ\big( \cP^{\cR_n(y)}(x) \big)}{|\cR_n(y)|}\,d\lambda(x).
\end{eqnarray*}
We define a new measure $\lambda^{*}$ on $X\setminus \cC_n(y)$ by setting
\[
\lambda^{*}(A) := \frac{\lambda(A)}{\lambda(X \setminus \cC_n(y))}.
\]
Then, since $\cC_n(y)$ is a disjoint union of atoms of the partition $\cP^{\cR_n(y)}$ of $X$, 
\begin{eqnarray*}
\int_{X \setminus \cC_n(y)} \frac{\cJ\big( \cP^{\cR_n(y)}(x) \big)}{|\cR_n(y)|}\,d\lambda(x)
&=& \lambda\big( X \setminus  \cC_n(y) \big) \, \int_{X \setminus \cC_n(y)} \frac{\cJ^{*}\big( \cP^{\cR_n(y)}(x) \big)}
{\big| \cR_n(y) \big|}\,d\lambda^{*}(x) \\
&& \quad - \lambda\big( X \setminus \cC_n(y) \big)\, \log\big( \lambda(X \setminus \cC_n(y)) \big), 
\end{eqnarray*}
where $\cJ^{*}$ denotes the information function with respect to $\lambda^{*}$. 
Now since the integral on the right hand side in the inequality above is just
the Shannon entropy of the partition $\cP^{\cR_n(y)}$ with respect to the 
measure $\lambda^{*}$, we arrive at
\begin{eqnarray*}
\int_{X \setminus \cC_n(y)} \frac{\cJ\big( \cP^{\cR_n(y)}(x) \big)}{|\cR_n(y)|}\,d\lambda(x) \leq 
\lambda\big( X \setminus \cC_n(y) \big)\,\log|\cP| - \lambda\big( X \setminus \cC_n(y) \big)\, 
\log\big( \lambda(X \setminus \cC_n(y)) \big).
\end{eqnarray*}
Clearly, by Theorem~\ref{thm:MAIN_SMBhyp}, for $\nu$-almost every $y\in Y$, 
the latter expression tends to zero as $n \to \infty$. 
Now sending $\varepsilon \to 0$ yields the first assertion of the claimed statement. 
For the second statement note that due to the dominated convergence theorem
(with dominating function $g(y):= 2\,|\cP|$) 
we have 
\begin{eqnarray*}
\lim_{n \to \infty} \int_Y \int_X \Big| \frac{\cJ\big(
\cP^{\cR_n(y)}(x) \big)}{|\cR_n(y)|} - h^{*}_{\cP} \Big|\,d\lambda(x)\,d\nu(y) = 0. 
\end{eqnarray*}
This concludes the proof of the corollary. 
\end{proof}

\section{Amenable relations, injective cocycles and ergodic extensions}\label{sec:ergodicity}

\subsection{Groups admitting injective cocycles}\label{sec:injective} 

As usual, Let $(Y,\nu)$ be a probability space, and  let $\cR\subset Y\times Y$ be a p.m.p.\@ Borel equivalence relation with $\cR$-invariant probability measure $\tilde{\nu}$, such that $\cR=\bigcup_{n\in \NN}\cR_n$ is hyperfinite, or equivalently, $\cR$ is amenable in the sense of \cite{CFW81}.  
Let $(X,\lambda)$ be a p.m.p.\@ action of a countable group $\Gamma$ and let $\alpha : \cR\to  \Gamma$ be a measurable cocycle. Let $\cR^X$ denote the extended relation on $(X\times Y,\lambda\times \nu)$. 

As note in the introduction, by a result of Danilenko \cite[Cor. 2.7]{Da01}, given the ingredients just listed the following limit exists 
\begin{eqnarray*}
h_{\cP}^{*}(\alpha) := \lim_{n \to \infty} \int_Y \frac{h^{\cP}\big( \cR_n \big)(y)}{\big| \cR_n(y)\big|}\, d\nu(y).
\end{eqnarray*}
A meaningful entropy invariant arises when we assume that the hyperfinite exhaustion is by bounded finite relations, and the cocycle is class injective. As we will see shortly, in fact {\it all} countable groups admit a class injective cocycle defined on a p.m.p. ergodic hyperfinite relation, but let us first note some familiar explicit examples for this property.


 
 {\bf 1) Amenable groups.}

Let $\Gamma$ be amenable, let $(Y,\nu)$ be any p.m.p.\@ action of $\Gamma$, and assume that the action is essentially free. For example, we can take $Y$ to the Bernoulli action of $\Gamma$ on $\set{0,1}^\Gamma$. 
  Let $\cR=\cO_\Gamma$ be the orbit relation of $\Gamma$ on $Y$, and then there is a cocycle $\alpha : \cR \to \Gamma$ given by $\alpha(\gamma y,y)=\gamma$. This cocycle is indeed an injective cocycle on a p.m.p.\@ amenable equivalence relation $\cR$, by amenability of $\Gamma$ and freeness of the action. Given a p.m.p.\@ action of $\Gamma$ on a space $(X,\lambda)$, clearly $\Gamma$ acts on the product $(X\times Y, \lambda\times \nu)$, and the orbit relation of $\Gamma$ in the product coincides with the extended relation $\cR^X$ on $X\times Y$ that we have used throughout the paper. Thus $\alpha :\cR\to \Gamma$ is an injective cocycle defined on a p.m.p.\@ amenable relation.

 {\bf 2) The Maharam extension of the Poisson boundary.}  For every countable infinite group $\Gamma$, and every generating probability measure $\mu$ on $\Gamma$, the Poisson boundary $B=B(\Gamma,\mu)$ is an amenable action of $\Gamma$, in the sense defined by Zimmer \cite{Zi78}, or equivalently, in the sense of \cite{CFW81}. Let $\eta$ denote the stationary measure on $B$, and $r_\eta(\gamma,b)=\frac{d\eta\circ \gamma}{d\eta}(b)$ the Radon-Nikodym derivative cocycle of $\eta$, so that $r_\eta : \Gamma \times B \to \RR^\ast_+$. The Maharam extension of $B$ by the cocycle $r_\eta$ is the $\Gamma$-action on $B\times \RR$ given by 
$\gamma(b,t)=(\gamma b, t-\log r_\eta(\gamma,b))$, and this action preserves the measure $\eta\times \theta$, where $d\theta(t)=e^t dt$ and $dt$ is Lebesgue measure on $\RR$. This action is again an amenable action of $\Gamma$, being an extension of an amenable action.
Let us define $Y=B\times (-\infty,0)$, and let $\nu$ be the restriction of $\eta\times L$ to $Y$, a finite 
measure which we normalize to be a probability measure. Then $(Y,\nu)$ is a probability space, and we define the relation $\cR$ on it to be the restriction of the orbit relation $\cO_\Gamma$ defined by $\Gamma$ on $B\times \RR$ to the subset $Y$. Thus $\cR$ is a p.m.p.\@ Borel equivalence relation with countable classes, since $\Gamma$ is countable and preserves the measure $\eta\times L$. The orbit relation $\cO_\Gamma$ is an amenable relation, hence it is hyperfinite, and as a result so is its restriction $\cR$ to the subset $Y$. Finally, if the action of $\Gamma$ on its Poisson boundary $B$ is essentially free, then we can define a cocycle 
$\alpha : \cO_\Gamma \to \Gamma$, by the formula $\alpha(\gamma (b,t),(b,t))=\gamma$. This cocycle is well-defined since the elements in a $\Gamma$-orbit are in bijective correspondence to the elements of $\Gamma$, by essential freeness. Restricting $\alpha$ to $\cR$ we obtain a cocycle from an amenable p.m.p.\@ equivalence relation $\cR$ to $\Gamma$. Furthermore, this cocycle is injective in this case, since 
$\gamma(b,t)=(\gamma b, t-\log r_\eta(\gamma,b))$, so that if $\gamma\neq \gamma^\prime$ then 
$\alpha((b,t), \gamma(b,t))\neq \alpha((b,t), \gamma^\prime(b,t))$. 

Clearly, the class of groups admitting a random walk such that the action on the associated Poisson boundary is essentially free is extensive indeed.  In fact typically many different random walks on a given non-amenable group $\Gamma$ give rise to Poisson boundaries admitting an essentially free action. It is a remarkable feature of the construction of  orbital Rokhlin entropy that it gives one and the same value for the entropy of the $\Gamma$-action on $X$, provided only that this action is ergodic and essentially free, and the value is independent of which cocycle $\alpha : \cR \to \Gamma$ as above was chosen to calculate it. 

{\bf 3)} We now note the following result, communicated to us by A. Kechris.
\begin{Proposition}\cite{Ke17} \label{kechris-1} For any countable infinite group $\Gamma$, there is a class injective cocycle from a hyperfinite p.m.p ergodic equivalence relation $E$ into $\Gamma$. 
\end{Proposition}
\begin{proof} Let $E_0$ be the equivalence relation on $2^{\NN}$ defined by the synchronous tail relation, namely $x E_0 y \iff \exists n \forall m>n\,\, (x_m = y_m)$. This is a hyperfinite relation consisting of an increasing union of Borel equivalence relations $\cR_n$ each with finite classes of equal size, and with respect to the usual measure on $2^{\NN}$, it is p.m.p and ergodic. Given any countable group $\Gamma$, consider the shift action of $\Gamma$ on $2^\Gamma$ and let $\mathcal{O}_\Gamma$ be the orbit relation determined by the action. Denote  by $E_\Gamma$ the induced equivalence relation on the free part of the action $Fr_\Gamma$, namely on the invariant full measure set consisting of points with trivial stabilizer in $\Gamma$. This is again p.m.p, ergodic with respect to the product measure, and in particular it is not measure-theoretically smooth, namely there is no Borel set meeting every class of $E_\Gamma$ in a single point. By a special case of the Glimm-Effros Dichotomy cf.\@ \cite[Theorem~28.6]{KM04}, there is then a Borel injection $f\colon 2^{\NN} \to Fr_\Gamma$ such that $x E_0 y \iff f(x) E_\Gamma f(y)$. The action of $\Gamma$ being free on $Fr_\Gamma$, this injection defines a Borel cocycle $\alpha(x,y)$ form $E_0$ into $\Gamma$, given by $\alpha(x,y) \cdot f(x) = f(y)$. This cocycle is clearly class injective.
\end{proof}

Thus, using the construction described in Lemma \ref{kechris-1}, Theorem~\ref{thm:MAINcocycleentropy} establishes that orbital entropy and orbital Rokhlin entropy can always be defined for a p.m.p.\@ action of a countable group. But furthermore,  we have seen above that for essentially free, ergodic actions of $\Gamma$, our notion of orbital Rokhlin entropy coincides with Rokhlin entropy no matter which class-injective cocycle has been chosen, and no matter which bounded hyperfinite exhaustion for the ergodic hyperfinite relation has been used. 


\subsection{Ergodicity of cocycle extensions}

The Shannon-McMillan-Breiman theorem stated in Theorem \ref{thm:MAIN_SMBhyp}, being a pointwise convergence result for orbital entropy of finite partitions, requires an additional ergodicity assumption for its validity, beyond those sufficient to guarantee the existence of orbital entropy itself. We note that an ergodicity assumption also plays a role in the Shannon-McMillan-Breiman theorem for amenable groups, see \cite{Li01}. To explain the ergodicity consition in question, let us first recall that in \cite[Def. 2.2]{BN15a} a notion of weak-mixing for a cocycle $\alpha$ on p.m.p.\@ relation $\cR$ on $(Y,\nu)$ was defined, as follows. A cocycle $\alpha :\cR \to \Gamma$ is weak-mixing if for {\it every} p.m.p.\@ ergodic action of $\Gamma$ on a spaces $(X,\lambda)$, the extended relation $\cR^X$ on $X\times Y$ is ergodic with respect to the product measure $\lambda \times \nu$. In particular, the relation $\cR$ itself must be ergodic. 
This definition is a natural extension of the notion of weak-mixing for group actions, where the action of $\Gamma$ on a space $(B,\eta)$ (with $\eta$ not necessarily invariant), is called weak-mixing if given any p.m.p.\@ ergodic action of $\Gamma$ on a space $(X,\lambda)$, the product action of $\Gamma$ on $(X\times B, \lambda\times\eta)$ is still ergodic. 

Let us now  mention some specific classes of groups and their actions, for which explicit information on the ergodicity of the cocycle extension can be provided. 

{\bf 1) Amenable groups.} When $\Gamma$ is amenable, referring to the relation $\cR$ and the cocycle $\alpha$ defined in \S \ref{sec:injective}(1), if the p.m.p. $\Gamma$-action on $(Y,\nu)$ is weak-mixing (in the usual sense for group actions), then the $\Gamma$-action on $(X\times Y,\lambda\times\nu)$ is ergodic for {\it every}  ergodic p.m.p.\@ action of $\Gamma$. 
Thus, referring to \S \ref{sec:injective}(1), the cocycle $\alpha : \cR \to \Gamma$ defined there is a weak-mixing cocycle and so here the extended relation $\cR^X$ is ergodic.

{\bf 2) Poisson boundaries. } When $\Gamma$ is non-amenable, the most important source (but not the only one) of weak-mixing actions of a countable group $\Gamma$ is the set of its actions on Poisson boundaries $B=B(\Gamma,\mu)$ \cite{AL05}. These actions are amenable as noted above, and satisfy a stronger condition than weak-mixing, namely double ergodicity with coefficients, see \cite{Ka03}. If $\Gamma$ is a non-amenable group, then the ergodic action on a Poisson boundary $(B,\eta)$ is not measure-preserving, and thus of  type $III$. The action on the Maharam extension $(B\times \RR,\eta\times L)$ is in fact measure-preserving, on a $\sigma$-finite (but not finite) measure space, but the Maharam extension is not necessarily an ergodic action of $\Gamma$.  The possibilities for it are determined by the type of the $\Gamma$-action on $(B,\eta)$, and in particular, if it is of type $III_1$, then the Maharam extension is ergodic. In general, this does not imply that for any ergodic action of $\Gamma$ on $(X,\lambda)$, the Maharam extension of $(X\times B,\lambda \times\eta)$ is ergodic.  If indeed this is the case for {\it every} p.m.p.\@ action of $\Gamma$, then the action of $\Gamma$ on $(B,\eta)$ is defined in \cite{BN13b} to have stable type $III_1$. 

Assume that the $\Gamma$-action on $(B,\eta)$ is essentially free, and let $(Y,\nu)$, $\cR$ and $\alpha$ be as defined  \S \ref{sec:injective}(2). Then the cocycle $\alpha :\cR\to \Gamma$ is injective, defined on a p.m.p.\@ amenable relation, and if the type of the $\Gamma$-action on $(B,\eta)$ is $III_1$ it is ergodic. If, furthermore, the $\Gamma$-action on $(B,\eta)$ has stable type $III_1$, then the cocycle $\alpha$ is weak-mixing, and hence $\cR^X$ is ergodic for every p.m.p.\@ action of $\Gamma$. Thus all the assumptions of the Shannon-McMillan-Breiman theorem are verified in this case, for which examples will be provided below.

{\bf 3) Non-trivial stable type.} If the type of the $\Gamma$-action on $(B,\eta)$ is $III_\tau$ for some $\tau > 0$ then the action of $\Gamma$ on the Maharam extension has a set of ergodic components admitting a free transitive action of  the circle group $\RR/\ZZ \cdot\log \tau $. The circle group acts on the Maharam extension and commutes with the $\Gamma$-action. If this is also the situation for the Maharam extension of all the spaces $(X\times B, \lambda\times \eta)$ for {\it every} ergodic p.m.p.\@ action of $\Gamma$, then the $\Gamma$-action on $(B,\eta)$ is defined in \cite{BN13b} to have stable type $\tau$. In that case, it is also possible to prove a version of the Shannon-McMillan-Breiman pointwise convergence theorem, which applies, rather than to the information functions we defined, to a further  average of them. In particular, this provides a proof of the Shannon-McMillan mean-convergence theorem in our context.  In the interest of brevity, however, we shall provide the details elsewhere. We refer to \cite{BN13b} and \cite{BN15a} for to a detailed discussion of type, stable type and Maharam extensions in the context of pointwise ergodic theorems for group actions.

{\bf 4) Ergodicity and mixing conditions.} Given an injective cocycle $\alpha : \cR \to \Gamma$ on an ergodic p.m.p.\@ amenable relation and a p.m.p.\@ action of $\Gamma$ on $(X,\lambda)$, it is possible to develop criteria to show that the extended relation $\cR^X$ is ergodic,   provided that the $\Gamma$-action on $X$ satisfies additional ergodicity or mixing conditions. This implies that the Shannon-McMillan-Breiman theorem is valid for $\Gamma$-actions on a suitable class of p.m.p.\@ actions on $(X,\lambda)$. 
A simple example of this phenomenon arises for the p.m.p.\@ actions of  finitely generated non-abelian free groups $\FF_r$. 
Here, taking the boundary $(\partial \FF_r, \nu)$ with the uniform measure $\nu$, we construct a cocycle $\alpha : \cR \to \FF_r$, where $\cR$ is an amenable p.m.p.\@ relation on $\partial \FF_r$, which is not ergodic, but in fact has exactly two ergodic components. If $(X,\lambda)$ is a p.m.p.\@ ergodic action of $\FF_r$ for which the index $2$ subgroup of even length words is ergodic, then the extended relation $\cR^X$ is an ergodic relation, see \cite{BN13a} for a detailed exposition. Hence these actions of $\FF_r$ satisfy the 
Shannon-McMillan-Breiman theorem.  We will give a detailed exposition of this case, which will also demonstrate the geometric significance of the theorem, in the next section.

%

To conclude this section let us  note the following results concerning type and stable type, which are relevant to the foregoing discussion. 
\begin{Examples}
\begin{enumerate}

 
\item Let $\Gamma$ be an irreducible  lattice in a connected semisimple Lie group with finite center and without compact factors. Then the action of $\Gamma$ on the maximal boundary $(G/P,m)$, where $P$ is a minimal parabolic subgroup and $m$ is the Lebesgue measure class, is amenable and has {\it stable type} $III_1$, see \cite{BN13b}. 
\item Let $\Gamma$ be a discrete non-elementary subgroup of isometries of real hyperbolic space. By 
\cite[Thm 6]{Su82}, the type of the action of $\Gamma$ on the boundary of hyperbolic space with respect to the Lebesgue measure class is $III_1$, when restricted to the recurrent part. In \cite{Sp87} this result was proved for the action of fundamental groups of compact connected negatively curved manifolds acting on the visual boundary with the manifold measure class. 

\item Let $\Gamma$ be a finitely generated free group. By \cite{RR97}, the action of $\Gamma$ has non-trivial type with respect to harmonic measures, namely stationary measures of suitable random walks. For  $\Gamma$ being a word hyperbolic group, it was proved in \cite{INO08} that the action of $\Gamma$ on the Poisson boundary associated with a generating measure of finite support has non-trivial type. 
\item In \cite{Bo14} it was proved that for word hyperbolic groups (with an additional technical condition) the action on the Gromov boundary with respect to a quasi-conformal measure (and in particular the Patterson-Sullivan measure) measure has non-trivial {\it stable type}.\end{enumerate}
\end{Examples}

\subsection{Ergodicity of the extended relation for strongly mixing actions}\label{kechris-Vaes}

We now turn to show that in fact the Shannon-McMillan-Breiman theorem holds for {\it all} free strongly mixing actions of {\it all} countable groups, 
provided the equivalence relation $\cR$ and the cocycle $\alpha$ are chosen suitably. We have seen in Theorem~\ref{thm:MAIN_SMBhyp} that it suffices to ensure that the extended relation is ergodic. Choosing $\cR$ and $\alpha$ suitably, this will indeed be the case, and this fact relies on the following result, communicated to us by A. Kechris.

\begin{Proposition}\cite{Ke17}\label{Kechris-2}
For every p.m.p. free action $(X,\mu)$ of a countable group $\Gamma$ which is strongly mixing, there exists a p.m.p. ergodic hyperfinite relation $\cR$ and a class injective cocycle $\alpha : \cR\to \Gamma$, such that the extended relation $\cR^X$ is ergodic.  
\end{Proposition}
\begin{proof}  Consider any ergodic p.m.p.\@ action $(X,\mu)$ of any countable group $\Gamma$, and the associated equivalence relation $\mathcal{O}_\Gamma$ determined by the orbits of $\Gamma$ in $X$. It is well known (see e.g. \cite[Theorem~3.5]{Ke10a}) that   $\mathcal{O}_\Gamma$ contains as a subrelation the orbit relation of a free measure-preserving action of a suitable element $T$ in the full group. Since $T$ is orbit equivalent to a free ergodic action of the group $\bigoplus_\ZZ \ZZ_2$ (e.g. by \cite{CFW81}), it follows that $\mathcal{O}_\Gamma$ also contains an ergodic hyperfinite p.m.p.\@ relation $\cR$.
Clearly, if the $\Gamma$-action on $X$ is free, then there is a class injective cocycle $\alpha : \cR\to \Gamma$, given by $\alpha (\gamma x,x)=\gamma$, whenever $(\gamma x,x)\in \cR$. As already noted, by \cite{CFW81} the relation $\cR$ is generated by a p.m.p. free action of $\ZZ$, and therefore using Dye's theorem \cite{Dy59} on the orbit equivalence of any two free p.m.p. actions of $\ZZ$, we can assume without loss of generality that $\cR$ is  generated by a strongly mixing free p.m.p. action of $\ZZ$ on $(X,\mu)$. Thus there exists a strongly mixing p.m.p.\@ map $T : X \to X$ whose orbits in $X$ are precisely the equivalence classes of $\cR$. The extended relation $\cR^X$ on $X\times X$ using the cocycle $\alpha$ that we have used throughout is then also generated by a p.m.p.  action of $\ZZ$ on $X\times X$, given by 
$T_X(x,x^\prime)=(Tx, \alpha (Tx,x)x^\prime)$. Now note that if the original $\Gamma$-action on $(X,\mu)$  is strongly mixing, namely the matrix coefficients vanish at infinity in $\Gamma$, then the extended action of $\ZZ$ on $X\times X$ given by the action of $T_X$ must be strongly mixing as well.  It suffices to check this fact on a dense subspace of functions in $L^2(X\times X)$, and we choose this dense subspace to consist of product functions $f(x)f^\prime(x^\prime)$, with $f,f^\prime \in L^2(X)\cap L^\infty(X)$.
Then 
$$\inn{(T_X)^n (ff^\prime),gg^\prime}=\int_{X\times X} f(T^nx)f^\prime(\alpha(T^n x,x)x^\prime)g(x)g^\prime(x^\prime)d\lambda(x)d\lambda(x^\prime)$$
$$=\int_{X} f(T^nx)g(x)\left(\int_X f^\prime(\alpha(T^n x,x)x^\prime)g^\prime(x^\prime)d\lambda(x^\prime)\right)d\lambda(x)$$
Note that for any given $T$-orbit in $X$, namely for each given class of $\cR$, the cocycle $\alpha$ is class injective, and hence the values of the cocycle $\alpha$ on that class must converge to infinity in $\Gamma$. It follows for almost any $x\in X$, $\alpha(T^n x,x)$ converges to infinity in $\Gamma$. By strong mixing of $\Gamma$ on $X$ the inner integral  converges to $\int_X f^\prime g^\prime d\lambda(x^\prime)$, and by Lebesegue dominated convergence theorem, the doubly integral then converges to 
$\int_{X\times X}ff^\prime \cdot  \int_X gg^\prime$, so that $T_X$ is strongly mixing. 

Finally, having shown that $T_X$ acting on $X\times X$ is strongly mixing, we can certainly conclude that it is ergodic, and it follows that the equivalence relation determined by the $T_X$-orbits in $X\times X$ is an ergodic relation. By construction this relation is precisely the extended relation associated $\cR^X$ associated with the cocycle $\alpha$, so that the extended relation is indeed ergodic. 

\end{proof}

\subsection{Ergodicity of the extended relation via weakly mixing Bernoulli actions}
In a recent paper \cite{VW17}, Vaes and Wahl show that there is a very large collection of countable groups $\Gamma$ which
admit a nonsingular Bernoulli action of type $III_1$. They verify this for instance for groups containing at least one
element of infinite order and conjecture that $\Gamma$ admits a nonsingular Bernoulli action $\Gamma \curvearrowright (B,\eta)$ of type $III_1$
if and only if the first $L^2$-cohomology of $\Gamma$ vanishes.  
These Bernoulli actions give rise to essentially free, p.m.p.\@ and {\it weakly mixing} Maharam extensions, 
say $\Gamma \curvearrowright (B \times \mathbb{R},\eta \times e^{t} dt)$,
see \cite[Theorems~5.1 and~6.1]{VW17}. The weak mixing condition implies that for any p.m.p. ergodic action of $\Gamma$ on $(X,\mu)$, the action of $\Gamma$ on $(X\times B \times \RR, \mu\times \eta\times e^t dt)$ is ergodic. Now restrict the orbit relation of $\Gamma$ in this action to the probability space $X\times Y$, where $Y:= B \times (-\infty,0)$ and $\nu := \eta \times e^{t} dt$, which we denote by 
$\cS=\cS(X,B)$. 

As we saw in the proof of Lemma \ref{kechris-1}  
the relation $\cS$
admits an ergodic, hyperfinite subequivalence relation $\cS_0=\cS_0(X,B)$. By freeness of the original $\Gamma$-action on $X$, we obtain a canonical cocycle defined over the hyperfinite p.m.p.\@ ergodic equivalence
relation $\cS_0$, given by  $\alpha: \cS_0 \to \Gamma$, where $\alpha\big( (x,y), (x^{\prime}, y^{\prime}) \big) = g$ for the
unique element $g \in \Gamma$ which satisfies $g \cdot (x^{\prime}, y^{\prime}) = (x,y)$. This way, the 
ergodicity assumption required for the validity of the convergence result appearing in the Shannon-McMillan-Breiman theorem, cf.\@ Theorem~\ref{thm:MAIN_SMBhyp},
is guaranteed to hold. Note however, that the present set-up is different than the set-up considered previously, since the relation $\cS_0$  and the cocycle $\alpha$ depend on $X$ and $B$. Here we do not establish  ergodicity
of the extended relation via properties of one fixed auxiliary equivalence relation $\cR$ and cocycle $\alpha$. Rather, we obtain the ergodicity of $\cS_0$ 
directly, as an equivalence relation on $(X \times Y, \mu \times  \nu)$, using weak-mixing of $B\times \RR$. This argument, unlike our previous one, cannot be used in conjunction with the subadditive principle, and therefore the limit in the convergence result that it provides is not shown to be independent of the cocycle $\alpha$ and the space $B$.


\section{The free group} \label{sec:freegroup}

\subsection{The boundary of the free group}
We briefly describe the ingredients we need for our analysis, following the exposition of \cite{BN13a}. 

Let $\FF=\langle a_1,\dots ,a_r \rangle$ be the free group of rank $r\ge 2$, with $S=\{a_i^{\pm 1}:~1\le i \le r\}$ a free set of generators. The (unique) {\em reduced form} of an element $g\in \FF$ is the expression  $g=s_1\cdots s_n$ with $s_i \in S$ and $s_{i+1}\ne s_i^{-1}$ for all $i$. Define $|g|:=n$, the length of the reduced form of $g$. The distance function on $\FF$ is defined by $d(g_1,g_2):=|g_1^{-1}g_2|$. The Cayley graph associated with the generating set $S$ is a regular tree of valency $2r$, and $d$ coincides with its edge-path distance.

The boundary of $\FF$ is the set of all sequences $\xi=(\xi_1,\xi_2,\ldots) \in S^\NN$ such that $\xi_{i+1} \ne \xi_i^{-1}$ for all $i\ge 1$. We denote it by $\partial \FF$. A metric $d_\partial$ on $\partial \FF$ is defined by $d_\partial\big( (\xi_1,\xi_2, \ldots), (t_1, t_2, \ldots) \big) = \frac{1}{n}$ where $n$ is the largest  natural number such that $\xi_i=t_i$ for all $i < n$. If $\{g_i\}_{i=1}^\infty$ is any sequence of elements in $\FF$ and $g_i:=t_{i,1}\cdots t_{i,n_i}$ is the reduced form of $g_i$ then $\lim_i g_i = (\xi_1,\xi_2,\ldots) \in \partial \FF$ if $t_{i,j}$ is eventually equal to $\xi_j$ for all $j$. If $\xi \in \partial \FF$ then we will write $\xi_i \in S$ for $i$-th element in the sequence  $\xi=(\xi_1,\xi_2,\xi_3,\ldots)$.

We define a probability measure $\nu$ on $\partial \FF$, by the requirement that every finite sequence $t_1,\ldots, t_n$ with $t_{i+1} \ne t_i^{-1}$ for $1\le i <n$, the following holds :
$$\nu\Big(\big\{ (\xi_1,\xi_2,\ldots) \in \partial \FF :~ \xi_i=t_i ~ 1\le i \le n\big\}\Big) := |S_n(e)|^{-1}=(2r-1)^{-n+1}(2r)^{-1}.$$

There is a natural action of $\FF$ on $\partial \FF$ by 
$$(t_1\cdots t_n)\xi := (t_1,\ldots,t_{n-k},\xi_{k+1},\xi_{k+2}, \ldots)$$ where $t_1,\ldots, t_n \in S$,  $t_1\cdots t_n$ is in reduced form and $k$ is the largest number $\le n$ such that $\xi_i^{-1} = t_{n+1-i}$ for all $i\le k$.  Observe that if $g=t_1 \cdots t_n$ then the Radon-Nikodym derivative satisfies
$$\frac{d\nu \circ g}{d\nu}(\xi) = (2r-1)^{2k-n}.$$

\subsection{The horospherical relation and the fundamental cocycle}\label{sec:free group review}
Let $\cR$ be the equivalence relation on $\partial \FF$ given by $(\xi,\eta) \in \cR$ if and only when writing $\xi=(\xi_1,\xi_2,\ldots)$ and $\eta=(\eta_1,\eta_2,\ldots)$, there exists $n$ such that $\eta_i = \xi_i$ for all $i> n$. Thus $\eta\cR \xi$ if and only if $\eta$ and $\xi$ have the same synchronous tail, if and only if they differ by finitely many coordinates only.

 Let $\cR_n$ be the equivalence relation given by $(\xi,\eta)\in \cR_n$ if and only if $\xi_i=\eta_i~\forall i > n$. Then $\cR$ is the increasing union of the finite subequivalence relations $\cR_n$. Thus $\cR$ is a hyperfinite relation.

Consider the relation $\cR^\prime$ on $\partial \FF$  such that $\eta\cR^\prime \xi$  if and only if there is a $g \in \FF$ such that $g\xi = \eta$ and $\frac{d\nu \circ g}{d\nu}(\xi) = 1$.  Note that the level set of the Radon-Nikodym derivative, namely $\set{g\in \FF\,;\, \frac{d\nu \circ g}{d\nu} (\xi) = 1}$ is the horosphere in the Cayley tree based 
at $\xi$ and passing through the identity in $\FF$. 
 Note that in that case $\xi=g^{-1}\eta$, and $\frac{d\nu\circ g^{-1}}{d\nu}(\eta)=\left(\frac{d\nu\circ g}{d\nu}(\xi)\right)^{-1}=1$, so that the relation is indeed symmetric. The transitivity of the horospherical relation $\cR^\prime$ follows from the cocycle identity which the Radon-Nikodym derivative satisfies. 
Thus $\cR^\prime$ is an equivalence relation, and by definition, the measure $\nu$ is $\cR^\prime$-invariant. 

The relation $\mathcal{R}^{\prime}$ on $\partial\mathbb{F}_{r}$ can also be defined more concretely by the condition that $\left(\xi,\eta\right)\in\mathcal{R}^{\prime}$ iff there
exists $k$ s.t. $\eta=g\xi$ and $g=\eta_{1}\cdot...\cdot\eta_{k}\cdot\xi_{k}^{-1}\cdot...\cdot\xi_{1}^{-1}$.
It follows that $\eta=g\xi$ has the same synchronous tail as $\xi$
from the $k+1$-th letter onwards. Equivalently, $g^{-1}$ belongs
to the horosphere based at $\xi$ and passing through the identity
in $\mathbb{F}_{r}$, namely the geodesic from $g^{-1}$ to $\xi$
and the geodesic from $e$ to $\xi$ meet at a point which is equidistant
from $e$ and $g^{-1}$. Thus it is natural to call $\mathcal{R}^{\prime}$ the \emph{horospherical
relation}
and the equivalence
class of $\xi$ under $\mathcal{R}_{n}$ the horospherical ball of
radius $n$ based at $\xi$. Since $\xi$ and $\eta=g\xi$ have the
same synchronous tail, $\mathcal{R}^{\prime}$ coincides with the
synchronous tail relation $\mathcal{R}$. 

The {\it fundamental cocycle} of the tail relation is the measurable map 
$\alpha : \cR \to \FF_r$ given, for $\eta=(\eta_1,\dots,\eta_k,\dots)$ 
and $\xi=(\xi_1,\dots,\xi_k,\dots)$ (with $(\eta,\xi)\in \cR_k$), by  
$$\alpha(\eta,\xi)=\eta_{1}\cdot...\cdot\eta_{k}\cdot\xi_{k}^{-1}\cdot...\cdot\xi_{1}^{-1}$$
so that $\alpha(\eta,\xi)\xi=\eta$. The cocycle takes values in the subgroup $\FF_r^{e}$ consisting of words of even length, and more precisely for each $k$ the set of values $\alpha(\eta, \xi)$ for fixed $\eta$ and $\xi\in\cR_k(\eta)$ coincides with the intersection of the word metric ball $B_{2k}(e)$ with the horoball  based at $\eta$ and passing through $e$. This set is called the horospherical ball of radius $2k$ determined by $\eta$ and denoted by $B_{2k}^\eta$. 

Note further that the set of cocycle values $\alpha(\xi, \eta)$ for $\xi\cR_k \eta$ with $\eta$ fixed, namely the set $\left(B_{2k}^\eta\right)^{-1}$, is given by 
$$\alpha(\xi,\eta)=\xi_{1}\cdot...\cdot\xi_{k}\cdot \eta^{-1}_{k}\cdot...\cdot\eta^{-1}_{1}$$
and this set
 contains words of length at most $2k$ whose first $(k-1)$ letters can be specified arbitrarily. 

Finally, consider the set of values of the cocycles $\alpha(\eta,\xi)$ for all points $\eta\neq \xi$ which are $\cR_k$ equivalent to another, i.e.\@ as we go over all pairs of distinct points in some equivalence class of the form 
$\cR_k(\zeta)$. This set of values clearly contains all the even words in a ball of radius $2k-2$, i.e. 
$B_{2k-2}(e)\cap\FF_r^e$.

The {\it finite order automorphisms} of $\cR$ is the subgroup $\Phi$ of $\cup_{n\in \NN} [\cR_n]$ generated by the transformations defined as follows.  Let $\pi_{n}:\partial\mathbb{F}_{r}\to S^{n}$ be the projection given by  $\pi_{n}\left(s_{1},s_{2},...\right)=\left(s_{1},s_{2},...,s_{n}\right)$.
We say that a map $\psi:\partial\mathbb{F}_{r}\to\partial\mathbb{F}_{r}$
has \emph{order $n$} if $\psi\left(\xi\right)=\psi\left(\xi^{\prime}\right)$
for any two boundary points $\xi,\xi^{\prime}\in\partial\mathbb{F}_{r}$
with $\pi_{n}\left(\xi\right)=\pi_{n}\left(\xi^{\prime}\right)$.

For any $\left(\xi,\xi^{\prime}\right)\in\mathcal{R}$ there exists
a map $\phi\in \Phi$ such that $\phi\left(\xi\right)=\xi^{\prime}$
and $\phi$ has order $n$ for some $n\in\mathbb{N}$, see \cite{BN13a}. 
Thus the group of finite order automorphisms clearly generates the synchronous tail (i.e.\@ the horospherical) relation.

{\it  The extended horospherical relation.} 
Let $\FF$ act by measure-preserving transformations on a probability space $(X,\lambda)$. Let $\cR_n^X$ be the equivalence relation on $ X\times \partial \FF$ defined by $((x,\xi),(x^\prime,\xi^\prime)) \in \cR_n^X$ if and only if there exists $g\in \FF$ with $(gx,g\xi)=(x^\prime,\xi^\prime)$ and $(\xi,\xi') \in \cR_n$ (i.e., if $\xi=(\xi_1,\ldots)\in S^\NN$ and $\xi'=(\xi'_1,\ldots)\in S^\NN$ then $\xi_i=\xi'_i$ for all $i\ge n$).

Inspecting the definitions, we see that the extended horospherical relation on $X\times \partial \FF_r$ coincides with the extension of the horospherical (i.e.\@ synchronous tail) relation $\cR$ on $\partial \FF_r$ via the fundamental cocycle $\alpha : \cR \to \FF_r$ defined above.

It is easy to see that for any $(y,\xi) \in  X\times \partial \FF$, 
$$\abs{\cR^X_n(y,\xi))} = \abs{\cR_n(\xi)} = (2r-1)^n.$$
The relation $\cR^X=\bigcup_{n \ge 1} \cR^X_n$ is thus a hyperfinite measurable equivalence relation, it preserves the measure $\nu\times \lambda$, and it is {\it uniform}, namely for each $n \ge 1$ almost every equivalence class of the relation $\cR_n^X$ has the same cardinality.

\subsection{Shannon-McMillan-Breiman theorem for the free groups}

Let $\mathcal{R}^{X}$ be the equivalence relation on $X\times \partial \FF_r$.
We may assume the action of $\mathbb{F}_{r}$ on $\left(X,\lambda\right)$
is ergodic, and we will use the following. 
\begin{Theorem}\cite{BN13a}\label{F^2-ergodic}. If $\mathbb{F}_{r}^{e}$ acts on $(X,\lambda)$
ergodically, then the diagonal action $\mathbb{F}_{r}^{e}\curvearrowright X\times \partial\mathbb{F}_{r}$
is ergodic. 
\end{Theorem}

\noindent{\bf The type of the boundary action.}
The type of the action $\FF_r\curvearrowright  (\partial \FF_r, \nu)$ is $III_\tau$ where $\tau=(2r-1)^{-1}$. It is follows from  \cite[Theorem 4.1]{BN13a} that the stable type of $\FF_r\curvearrowright  (\partial \FF_r, \nu)$ is $III_{\tau^2}$. In fact, if $\FF_r^e$ denotes the index 2 subgroup of $\FF$ consisting of all elements of even word length then $\FF_r^e\curvearrowright  (\partial \FF_r,\nu)$ is  of type $III_{\tau^2}$ and stable type $III_{\tau^2}$. It is also weakly mixing. Indeed, $(\partial \FF_r,\nu)$ is naturally identified with $B(\FF_r,\mu)$, the Poisson boundary of the random walk generated by the measure $\mu$ that is distributed uniformly on $S$.  By \cite{AL05}, the action of any countable group $\Gamma$ on the Poisson boundary $B(\Gamma,\kappa)$ is weakly mixing whenever the measure $\kappa$ is adapted. This shows that $\FF_r\curvearrowright  (\partial \FF_r,\nu)$ is weakly mixing. Moreover, if we denote $S^2 = \{st:~s, t \in S\}$, then $(\partial \FF_r,\nu)$ is naturally identified with the Poisson boundary $B(\FF_r^e, \mu_2)$ where $\mu_2$ is the uniform probability measure on $S^2$. So the action $\FF_r^e \curvearrowright  (\partial \FF,\nu)$ is also weakly mixing.


Let $\FF_r$ act on $\partial \FF_r \times \ZZ$ by $g(b,t) = (gb, t-\log r_{\nu}(g,b))$, which gives the discrete Maharam extension in this case. 
Let $\cR$ be the orbit-equivalence relation restricted to $\partial \FF_r \times \{0\}$, which we may, for convenience, identify with $\partial \FF_r$. In other words, $b \cR b'$ if and only if there is an element $g \in \FF_r$ such that $gb=b'$ and $\frac{d\nu \circ g}{d\nu}(b)=1$. As noted above, this is the same as the (synchronous) tail-equivalence relation on $\FF_r$. In other words, two elements $\xi=(\xi_1,\xi_2,\ldots)$, $\eta=(\eta_1,\eta_2,\ldots) \in \partial \FF_r$ are $\cR$-equivalent if and only if there is an $m$ such that $\xi_n=\eta_n$ for all $n \ge m$. But now note that if $b\cR (gb)$, then necessarily $g \in \FF_r^e$. So $\cR$ can also be r

Let $\alpha:\cR \to \FF_r^e$ be the cocycle $\alpha(gb,b)=g$ for $g\in \FF_r^e, b \in \partial \FF_r$. This is well-defined almost everywhere because the action of $\FF_r^e$ is essentially free. As noted above, given two $\cR_k$-equivalent points $\xi,\eta\in \partial \FF_r$ as above, the cocycle value  is given by 
$\alpha(\eta,\xi)=\eta_{1}\cdot...\cdot\eta_{k}\cdot\xi_{k}^{-1}\cdot...\cdot\xi_{1}^{-1}$. 

Because $\FF_r^e \curvearrowright  (\partial \FF_r,\nu)$ has type $III_{\tau^2}$ and stable type $III_{\tau^2}$, this cocycle is weakly mixing for $\FF_r^e$. In other words, if $\FF_r^e \curvearrowright (X,\mu)$ is any ergodic p.m.p action, then the equivalence relation $\cR^X$ defined on $X\times \partial \FF_r $ by the cocycle extension is ergodic. Since the relation is hyperfinite and the cocycle is injective, we conclude that the action of $\FF_r$ on any ergodic p.m.p.\@ space satisfies the Shannon-McMillan-Breiman theorem, provided only that $\FF_r^e$ acts ergodically on $X$. Applying Theorem  \ref{thm:MAIN_SMBhyp}, we conclude :   
\begin{Theorem}\label{free gps2} 
Let $(X,\mu)$ be any essentially free p.m.p. action of the free group $\FF_r$, $r\ge 2$. Assume that the action is ergodic under the index $2$ subgroup of even words $\FF_r^e$. Fix any finite partition $\cP$ of $X$. For a boundary point $\eta \in \partial \FF_r$, consider the sequence of partitions obtained from $\cP$ by refining it along the horospherical balls $B_{2k}^\eta$ determined by $\eta$ : 
$$\bigvee_{g\in B_{2k}^\eta}g \cP=\bigvee_{\xi\in \cR_k(\eta)}\alpha(\xi,\eta)^{-1}\cP\,.$$
Then the information functions of the refined partitions converge to the orbital entropy (of the partition $\cP$) for the action of $\FF_r$ on $X$, for $\mu$-almost every $x\in X$, for $\nu$-almost every $\eta \in \FF_r$. 
\end{Theorem}

\end{document}